\documentclass[a4paper, 12pt]{amsart}
\usepackage{amsmath}
\usepackage{amssymb, latexsym, amsthm}
\usepackage[all]{xy}

\newtheorem{Thm}{Theorem}[section]
\newtheorem{Prop}[Thm]{Proposition}
\newtheorem{Cor}[Thm]{Corollary}
\newtheorem{Lem}[Thm]{Lemma}

\theoremstyle{definition}
\newtheorem*{definition}{Definition}
\newtheorem*{remark}{Remark}

\numberwithin{equation}{section}

\begin{document}

\newcommand{\Coim}{\mathrm{Coim}}
\newcommand{\Z}[0]{\mathbb{Z}}
\newcommand{\Q}[0]{\mathbb{Q}}
\newcommand{\F}[0]{\mathbb{F}}
\newcommand{\N}[0]{\mathbb{N}}
\renewcommand{\O}[0]{\mathcal{O}}
\newcommand{\p}[0]{\mathfrak{p}}
\newcommand{\m}[0]{\mathrm{m}}
\newcommand{\Tr}{\mathrm{Tr}}
\newcommand{\Hom}[0]{\mathrm{Hom}}
\newcommand{\Gal}[0]{\mathrm{Gal}}
\newcommand{\Res}[0]{\mathrm{Res}}
\newcommand{\id}{\mathrm{id}}
\newcommand{\cl}{\mathrm{cl}}
\newcommand{\mult}{\mathrm{mult}}
\newcommand{\adm}{\mathrm{adm}}
\newcommand{\tr}{\mathrm{tr}}
\newcommand{\pr}{\mathrm{pr}}
\newcommand{\Ker}{\mathrm{Ker}}
\newcommand{\ab}{\mathrm{ab}}
\newcommand{\sep}{\mathrm{sep}}
\newcommand{\triv}{\mathrm{triv}}
\newcommand{\alg}{\mathrm{alg}}
\newcommand{\ur}{\mathrm{ur}}
\newcommand{\Coker}{\mathrm{Coker}}
\newcommand{\Aut}{\mathrm{Aut}}
\newcommand{\Ext}{\mathrm{Ext}}
\newcommand{\Iso}{\mathrm{Iso}}
\newcommand{\M}{\mathcal{M}}
\newcommand{\GL}{\mathrm{GL}}
\newcommand{\Fil}{\mathrm{Fil}}
\newcommand{\an}{\mathrm{an}}
\renewcommand{\c}{\mathcal }
\newcommand{\W}{\mathcal W}
\newcommand{\R}{\mathcal R}
\newcommand{\crys}{\mathrm{crys}}
\newcommand{\st}{\mathrm{st}}
\newcommand{\CM}{\mathrm{CM\Gamma }}
\newcommand{\CV}{\mathcal{C}\mathcal{V}}
\newcommand{\G}{\mathrm{G}}
\newcommand{\Map}{\mathrm{Map}}
\newcommand{\Sym}{\mathrm{Sym}}
\newcommand{\Spec}{\mathrm{Spec}}
\newcommand{\Gr}{\mathrm{Gr}}
\newcommand{\I}{\mathrm{Im}}
\newcommand{\Frac}{\mathrm{Frac}}
\newcommand{\LT}{\mathrm{LT}}
\newcommand{\Alg}{\mathrm{Alg}}
\newcommand{\MG}{\mathrm{M\Gamma }}
\newcommand{\To}{\longrightarrow}
\newcommand{\md}{\mathrm{mod}}
\newcommand{\MF}{\mathrm{MF}}
\newcommand{\CMF}{\mathcal{M}\mathcal{F}}
\newcommand{\Aug}{\mathrm{Aug}}
\renewcommand{\c}{\mathcal }
\newcommand{\uL}{\underline{\mathcal L}}
\newcommand{\Md}{\mathrm{Md}}
\newcommand{\wt}{\widetilde}
\newcommand{\op}{\mathrm}
\newcommand {\w}{\op{wt}}
\newcommand{\Ad}{\op{Ad}}
\newcommand{\ad}{\op{ad}}
\newcommand{\D}{\mathcal D}
\newcommand{\fr}{\mathfrak}

\title [Automorphisms of local fields]
{Automorphisms of local fields of period $p^M$ and nilpotent class $<p$}
\author{Victor Abrashkin}
\address{Department of Mathematical Sciences, 
Durham University, Science Laboratories, 
South Rd, Durham DH1 3LE, United Kingdom \ \&\ Steklov 
Institute, Gubkina str. 8, 119991, Moscow, Russia
}
\email{victor.abrashkin@durham.ac.uk}
\date{May 23, 2016}
\keywords{} 
\medskip 

\begin{abstract}  Suppose $K$ is a finite extension of $\mathbb{Q}_p$ 
containing a $p^M$-th primitive root of unity. For $1\leqslant s<p$ denote by 
$K[s,M]$ the maximal $p$-extension of $K$ with the Galois group of period 
$p^M$ and nilpotent class $s$. We apply the nilpotent Artin-Schreier theory 
together with the theory of the field-of-norms functor to give an explicit 
description of the Galois groups $\op{Gal}(K[s,M]/K)$. As application we 
prove that the ramification subgroup $\Gamma ^{(v)}_K$ of the absolute Galois group 
of $K$ acts trivially on $K[s,M]$ if and only if $v>e_K(M+s/(p-1))-(1-\delta _{1s})/p$, 
where $e_K$ is the ramification index of $K$ and $\delta _{1s}$ is the Kronecker symbol. 
\end{abstract}
\maketitle

\section*{Introduction} 
\medskip 

Everywhere in the paper $M\in\N $ is fixed and $p\ne 2$ is prime.  

Let $K$ be a complete discrete valuation field of characteristic $0$ 
with finite residue field $k\simeq\F _{q_0}$, where $q_0=p^{N_0}$, $N_0\in\N $. 
Fix an algebraic closure $\bar K$ of $K$ 
and denote by $K_{<p}(M)$ the maximal $p$-extension 
of $K$ in $\bar K$ with the Galois group of 
nilpotent class $<p$ and exponent $p^M$. 
Then $\Gamma _{<p}(M):=\op{Gal}(K_{<p}(M)/K)=\Gamma /\Gamma ^{p^M}C_p(\Gamma )$, 
where $\Gamma =\op{Gal}(\bar K/K)$ and $C_p(\Gamma )$ is the closure of the  
subgroup of commutators of order $\geqslant p$. 

Let \ $\{\Gamma ^{(v)}\}_{v\geqslant 0}$ be the ramification filtration of 
$\Gamma $ in upper numbering \cite{Se}. The importance of this additional 
structure on the Galois group $\Gamma $ (which reflects arithmetic properties of 
$K$) can be illustrated by the local analogue 
of the Grothendieck Conjecture \cite{Mo, Ab4, Ab5}: the knowledge of 
$\Gamma $ together with the filtration $\{\Gamma ^{(v)}\}_{v\geqslant 0}$ 
is sufficient to recover uniquely the isomorphic class of $K$ 
in the category of complete discrete valuation fields. 

Let  
$\{\Gamma _{<p}(M)^{(v)}\}_{v\geqslant 0}$ 
be the induced ramification filtration of $\Gamma _{<p}(M)$. Then the problem of  
arithmetical description of $\Gamma _{<p}(M)$ is the problem 
of explicit description of the filtration $\{\Gamma _{<p}(M)^{(v)}\}_{v\geqslant 0}$ 
in terms of generators of $\Gamma _{<p}(M)$. 

An analogue of this problem was studied in \cite{Ab1,Ab2,Ab3} in the case of local fields  
$\c K$ of characteristic $p$ with residue field $k$. More precisely, let  
$\c G=\op{Gal} (\c K_{sep}/\c K)$ and $\c G _{<p}(M)=
\c G/\c G^{p^M}C_p(\c G)$. In \cite{Ab1, Ab2} we developed 
a nilpotent version of the Artin-Schreier theory which allows us to 
construct identification of profinite groups 
$\c G _{<p}(M)=G(\c L)$. Here $\c L$ is a profinite Lie 
$\Z /p^M $-algebra of nilpotent class $<p$ and $G(\c L)$ 
is the pro-$p$-group, obtained from $\c L$ by the Campbell-Hausdorff composition law, 
cf. Subsection \ref{S1.2} below for more details and Subsection 1.1 in 
\cite{Ab10} for 
non-formal comments about nilpotent Artin-Schreier theory.

On the one hand, the above identification of $\c G_{<p}(M)$ with $G(\c L)$ depends 
on a choice of uniformising element in $\c K$ and, therefore, is not functorial 
(in particular, it can't be used directly to develop a nilpotent analog of 
classical local class field theory). On the other hand, the 
ramification subgroups $\c G_{<p}(M)^{(v)}$ can be now described in terms of 
appropriate ideals $\c L^{(v)}$ of the Lie algebra $\c L$. 
The definition of these ideals essentially uses the 
extension of scalars $\c L_k:=\c L\otimes W_M(k)$ of $\c L$ (such operation 
does not exist in the category of $p$-groups) 
together with the appropriate explicit system of generators 
of $\c L_k$, cf. Subsection \ref{S1.4}.  
This justifies the advantage of the language of Lie algebras in the theory of 
$p$-extensions of local fields.

In this paper we apply the above characteristic $p$ results to the study 
of similar properties in the mixed characteristic case, i.e. to the study 
of the group $\Gamma _{<p}(M)$ together with its ramification filtration.  
 Our main tool is 
the Fontaine-Wintenberger theory of the field-of-norms functor \cite{Wi}. 
Note also that we assume that $K$ contains a 
primitive $p^M$-th root of unity  and our methods 
generalize the approach from \cite{Ab11} 
where we considered the case $M=1$.  In some sense our theory can be treated 
as nilpotent version of Kummer's theory in the context of 
complete discrete valuation fields. 
As a result, we identify $\Gamma _{<p}(M)$ with the group $G(L)$, 
where  
$L$ is a Lie $\Z /p^M$-algebra and for an appropriate ideal $\c J$ of $\c L$, 
we have the following exact sequence  of Lie algebras 
\begin{equation}\label{E1} 
0\To\c L/\c J\To L\To C_M \To 0\, .
\end{equation}
Here $C_M$ is a cyclic group of order $p^M$ with the trivial structure of Lie algebra 
over $\Z /p^M$.  

As a first step in the study 
of $L$, we give an explicit description of the ideal $\c J$. 
More generally, if $C_s(L)$ is the closure of the ideal of commutators 
of order $\geqslant s$ 
in $L$, then for $s\geqslant 2$, we have $C_s(L)\subset\c L/\c J$ 
and exact sequence \eqref{E1} induces the exact sequences 
$$0\To\c L/\c L(s)\To L/C_s(L)\To C_M \To 0,$$
where all $\c L(s)$ are ideals in $\c L$. The main result of Section \ref{S3}, 
Theorem \ref{T3.3},  
describes these ideals $\c L(s)$ with $2\leqslant s\leqslant p$ 
and gives in particular that  $\c J=\c L(p)$. 

Extension \eqref{E1} splits in the category of $\Z /p^M$-modules and  
its structure can be 
given by explicit construction of a lift $\tau _{<p}$ of a generator of $C_M$ to $L$ 
and the appropriate differentiation  
$\op{ad}\tau _{<p}\in\op{End}(\c L/\c J)$. The study of 
$\op{ad}\tau _{<p}$ will be done in the next paper 
via methods used in the case $M=1 $ in \cite{Ab11}. 

In Section \ref{S4} we apply our approach to find for $1\leqslant s<p$, 
the maximal upper ramification numbers $v(K[s,M]/K)$ of the maximal 
extensions $K[s,M]$ of $K$ with Galois groups of period $p^M$ and nilpotent class $s$. 
(The maximal upper ramification number for a finite extension $K'/K$ in $\bar K$ 
is the maximal $v_0$ such that 
the ramification subgroups $\Gamma ^{(v)}$ act trivially on $K'$ if $v>v_0$.) 
This result can be stated 
in the following form, cf. Theorem \ref{Th4.4} from Section \ref{S4}:
\medskip 

$\bullet $\ {\it If $[K:\Q _p]<\infty $ and $\zeta _M\in K$ then for $1\leqslant s<p$, 
$$v(K[s,M]/K)=e_K\left (M+\frac{s}{p-1}\right )-\frac{1-\delta _{s1}}{p}\,.$$
where $e_K$ is the ramification index of $K/\Q _p$ and $\delta $ is the Kronecker symbol.}
\medskip 

\remark The case $s=1$ is very well-known and can be established without 
the assumption $\zeta _M\in K$. Is it possible to remove 
this restriction when $s>1$? 
\endremark 
\medskip 

{\bf Notation}. If $\mathfrak M$ is an $R$-module then its extension of scalars 
$\mathfrak M\otimes _RS$ will be very often denoted by $\mathfrak M_S$, 
cf. also another agreement in 
Subsection \ref{S1.1}. Very often we drop off the indication to $M$ 
from our notation and use just 
$K_{<p}, \Gamma _{<p}, \c G_{<p}$ etc. instead of $K_{<p}(M), 
\Gamma _{<p}(M), \c G_{<p}(M)$, etc. 

\medskip 

\section{ Preliminaries } \label{S1}
\medskip 

Let $\c K$ be a complete discrete valuation field of characteristic $p$ 
with residue field $k\simeq\F _{q_0}$, $q_0=p^{N_0}$, and fixed uniformiser $t_0$. 
In other words, $\c K=k((t_0))$. 

As earlier, $\c G=
\op{Gal}(\c K_{sep}/\c K)$, $\c K_{<p}=\c K_{<p}(M)$ 
is the subfield of $\c K_{sep}$ fixed by 
$\c G^{p^M}C_p(\c G)$ and 
$\c G_{<p}=\c G_{<p}(M)=\op{Gal}(\c K_{<p}/\c K)$. 
The ramification filtration of $\c G_{<p}$ 
was studied in details in \cite{Ab1, Ab2, Ab3}. 
We overview these results in the next subsections. 

\subsection { Compatible system of lifts modulo $p^M$} \label{S1.1}

The uniformizer $t_0$ of $\c K$ gives a $p$-basis for any separable 
extension $\c E$ of $\c K$, i.e. $\{1, t_0, \dots ,t_0^{p-1}\}$ 
is a basis of the $\c E^p$-module $\c E$. 
We can use $t_0$ to construct 
a functorial on $\c E$ (and on $M$) 
system of lifts $O(\c E)(=O_M(\c E))$ of   
$\c E$ modulo $p^M$. Recall that these lifts appear in the form 
$W_M(\sigma ^{M-1}\c E)[t]$, where 
$W_M$ is the functor of Witt vectors of length $M$, 
$\sigma $ is the Frobenius morphism of taking $p$-th power 
and ${t}=(t_0,0,\dots ,0)\in W_M(\c K)$. 

Note that ${t}\in O(\c K)\subset W_M(\c K)$,   
${t}\,\op{mod}\,p=t_0$ and 
$\sigma {t}={t}^p$.  The lift $O(\c K)$ is naturally identified 
with the algebra of formal Laurent series $W_M(k)(({t}))$ in the 
variable ${t}$ with coefficients in $W_M(k)$. 
A lift $\sigma $ of the absolute Frobenius endomorphism of $\c K$ 
to $O(\c K)$ is uniquely determined by the condition $\sigma {t}=
{t}^p$. For a separable extension $\c E$ of $\c K$ we then have an  
extension of the Frobenius $\sigma $ from 
$\c E$ to $O(\c E)(=W_M(\sigma ^{M-1}\c E)[t])$. As a result, 
we obtain a compatible system of lifts of the Frobenius endomorphism 
of $\c K_{sep}$ to $O(\c K_{sep})=\underset{\c E}\varinjlim O(\c E)$. 
For simplicity, we shall denote this lift also by $\sigma $. Note that $\sigma $ 
is induced by the standard Frobenius endomorphism $W_M(\sigma )$ of 
$W_M(\c K_{sep})\supset O(\c K_{sep})$. 

Suppose $\eta _0\in \Aut\,\c K$ and let $W_M(\eta _0)$ be the induced 
automorphism of $W_M(\c K)$. If $W_M(\eta _0)(t)\in O(\c K)$ then 
$\eta :=W_M(\eta _0)|_{O(\c K)}$ is a lift of $\eta _0$ to $O(\c K)$, i.e. 
$\eta\in\Aut\,O(\c K)$ and $\eta \,\op{mod}\,p=\eta _0$. 
With the above notation and assumption 
(in particular, $\eta (t)\in O(\c K)$) we have even more.

\begin{Prop} \label{P1.1}   
Suppose 
$\c E$ is separable over $\c K$, $\eta _{\c E0}\in\op{Aut}\,\c E$ and  
$\eta _{\c E0}|_{\c K}=\eta _0$. Then $\eta _{\c E}:=W_M(\eta _{\c E0})|_{O(\c E)}$ 
is a lift of $\eta _{\c E0}$ to $O(\c E)$ 
such that $\eta _{\c E}|_{O(\c K)}=\eta $.  
\end{Prop} 

\begin{proof} Indeed, using that  $O(\c E)=W_M(\sigma ^{M-1}\c E)[t]$,  we obtain 
$$\eta _{\c E}(W_M(\sigma ^{M-1}\c E))=W_M(\eta _{\c E0})
(W_M(\sigma ^{M-1}\c E))\subset W_M(\sigma ^{M-1}\c E))\subset O(\c E)\, ,$$
and $\eta _{\c E}(t)=W_M(\eta _{\c E0})(t)=W_M(\eta _0)(t)\in O(\c K)\subset O(\c E)$. 
So, $\eta _{\c E}(O(\c E))\subset O(\c E)$. Obviously, 
$\eta _{\c E}\,\op{mod}\,p=\eta _{\c E0}$. 
\end{proof}

\begin{remark} The above lifts $\eta _{\c E}$ commute with $\sigma $ if and only if 
$\eta $ commutes with $\sigma $, i.e. 
$\sigma (\eta (t))=\eta (t^p)$. In particular, if $\eta (t)=t\alpha ^{p^{M-1}}$ 
with $\alpha\in O(\c K)$ then 
$\sigma (\eta (t))=t^p\alpha ^{p^M}=\eta (t^p)$ 
(use that $\sigma (\alpha )\equiv \alpha ^p\,\op{mod}\,pO(\c K)$). 
\end{remark} 

A very special case of the above proposition appears as the following property: 

--- if $\c E/\c K$ is Galois then the elements $g$ of the group $\op{Gal}(\c E/\c K)$ 
can be naturally lifted to (commuting with $\sigma $) automorphisms of 
$O(\c E)$  
via setting $g(t)=t$. 
Therefore,  $O(\c K_{sep})$ has a natural structure of a $\c G$-module, 
the action of $\c G$ commutes with $\sigma $, $O(\c K_{sep})^{\c G}=O(\c K)$ and 
$O(\c K_{sep})|_{\sigma =\op{id}}=W_M(\F_p)$. 
\medskip 

Everywhere below we shall use the following simplified notation. 

{\bf Notation.}
 If $\mathfrak M$ is a $\Z /p^M$-module and $\c E$ is a separable extension of $\c K$ 
we set $\mathfrak M_{\,\c E}:=\mathfrak M _{O(\c E)}
(=\mathfrak M\otimes _{\Z /p^M}O(\c E)$). Similarly, we agree that 
$\mathfrak M_k:=\mathfrak M\otimes _{\Z/p^M}W_M(k)$. 
\medskip

\subsection{ Categories of $p$-groups and 
Lie $\Z/p^M$-algebras, \cite{Kh, La}} \label{S1.2}

If $L$ is a Lie $\Z/p^M$-algebra of nilpotent class $<p$, 
denote by $G(L)$ the $p$-group obtained from $L$ via the Campbell-Hausdorff 
composition law $\circ $ defined for $l_1,l_2\in L$ via  
$\wt{\exp}(l_1\circ l_2)=\wt{\exp}l_1\cdot \wt{\exp}l_2$. Here    
$$\wt{\exp}(x)=1+x+\dots +x^{p-1}/(p-1)!\, $$ 
is 
the truncated exponential from $L$ to the quotient of the enveloping 
algebra $\c A$ of $L$ modulo the $p$-th power of 
its augmentation ideal $J$. (This construction of the 
Campbell-Hausdorff operation was introduced in \cite{Ab1}, Subsection 1.2.) 

The correspondence $L\mapsto G(L)$ induces equivalence of the categories  
of finite Lie $\Z/p^M$-algebras and 
finite $p$-groups  
of exponent $p^M$ of the same nilpotent class $1\leqslant s_0<p$.
This equivalence can be extended to the similar categories 
of profinite Lie algebras and groups.  
\medskip 

\subsection{Witt pairing and Hilbert symbol, \cite{AJ, Fo}} \label{S1.3}

Let 
$$E(\alpha ,X)=\op{exp}\left (\alpha X+
\frac{\sigma (\alpha )X^p}{p}+\dots +\frac{\sigma ^n(\alpha )X^{p^n}}{p^n}
\dots \right )\in W(k)[[X]],$$
where $\alpha\in W(k)$, be the Shafarevich version of the Artin-Hasse exponential. Set 
$\Z^+(p)=\{a\in\N\ |\ \op{gcd}(a,p)=1\}$. Then  
any element $u\in\c K^*\op{mod}\,\c K^{*p^M}$ can be uniquely written 
as 
$$u=t_0^{a_0}\prod_{a\in\Z^+(p)}
E(\alpha _a,t_0^{a})^{1/a}\op{mod}\,\c K^{*p^M},$$
where $a_0=a_0(u)\in\Z\,\op{mod}\,p^M$ and all  
$\alpha _a=\alpha _a(u)\in W(k)\,\op{mod}\,p^M$. 

Let $\mathfrak M$ be a profinite free $W_M(k)$-module with the set of generators 
$\{D_0\}\cup\{D_{an}\ |\ a\in\Z^+(p),n\in\Z/N_0\}$. 
Use the correspondences 
\begin{equation} \label{E1.5} t_0\mapsto D_0,
\ \ \ E(\alpha ,t_0^a)^{1/a}\mapsto\sum_{n\,\op{mod}N_0} 
\sigma ^n(\alpha )D_{an},
\end{equation} 
to identify $\c K^*/\c K^{*p^M}$ 
with a closed $\Z /p^M$-submodule in $\mathfrak M$. Under this identification we have 
$\c K^*/\c K^{*p^M}\otimes_{\Z /p^M}W_M(k)=\mathfrak M$. 

Define the continuous action of the group $\langle\sigma\rangle =
\op{Gal}(k/\F_p)$ on $\mathfrak M$ as an extension of 
the natural action on $W_M(k)$ by setting 
$\sigma D_0=D_0$ and $\sigma D_{an}=D_{a,n+1}$. Then 
$\c K^*/\c K^{*p^M}=\mathfrak M^{\op{Gal}(k/\F_p)}$. 

The Witt pairing 
$$O(\c K)/(\sigma -\op{id})O(\c K)\times 
\c K^*/\c K^{*p^M}\To \Z /p^M,$$
is given explicitly by the symbol $[f,g)=\op{Tr}\left (
\op{Res}(f\,d_{\op{log}}\op{Col}\,g)\right )$. 
Here $\op{Tr}:W_M(k)\To \Z /p^M$ is induced by the trace of 
the field extension 
$k/\F _p$, 
$f\in O(\c K)$  and 
$\op{Col}\,g$ is the image of $g\in \c K^*/\c K^{*p^M}$ under the group homomorphism  
$\op{Col}:\c K^*/\c K^{*p^M}\longrightarrow O^*_M(\c K)$ uniquely defined 
on the above free generators of $\c K^*/\c K^{*p^M}$  via  the conditions 
$t_0\mapsto {t}$ and $E(\alpha ,t_0^a)\mapsto E(\alpha ,{t}^a)$. 
The Witt pairing is non-degenerate and  determines the identification 
$$\c K^*/\c K^{*p^M}=\op{Hom}_{\op{cont}}
(O(\c K)/(\sigma -\op{id})O(\c K),\Z /p^M).$$
It also  coincides with the Hilbert symbol 
(in the case of local fields of characteristic $p$) and allows us to specify 
explicitly the reciprocity map $\kappa :\c K^*/\c K^{*p^M}\longrightarrow\c G_{<p}^{ab}$ 
of class field theory. Namely, in the above notation we have 
$\kappa (g)f=f+[f,g)$.

\subsection{ Lie algebra $\c L$ and identification $\eta _M$} \label{S1.4}

Let $\wt{\c L}$ be a free profinite Lie $\Z /p^M$-algebra 
with the module of (free) generators $\c K^*/\c K^{*p^M}$. Then the  
$W_M(k)$-module  $\wt{\c L}_k$ has the  
set of free generators 
\begin{equation}\label{Eq1.2}
\{D_0\}\cup\{D_{an}\ |\ a\in\Z^+(p),n\in\Z/N_0\}.
\end{equation}

If $C_p(\wt{\c L})$ is the closure of the ideal of commutators of order 
$\geqslant p$, then 
$\c L=\wt{\c L}/C_p(\wt{\c L})$ 
is the maximal quotient of $\wt{\c L}$ of nilpotent class $<p$.

\begin{remark} $\c L_k$ is a free object in the category of profinite 
Lie $W_M(k)$-algebras of nilpotent class $<p$ with the set of free generators \eqref{Eq1.2}. 
\end{remark} 

We shall use the same notation 
$D_0$ and $D_{an}$ for the images of the elements of \eqref{Eq1.2} in $\c L$.
 Choose $\alpha _0\in W_M(k)$ such that 
$\op{Tr}\,\alpha _0=1$.  

Consider 
$e=\alpha _0D_0+\sum_{a\in\Z^+(p)} {t}^{-a}D_{a0}
\in G(\c L_{\c K})$. If we set 
$D_{0n}:=(\sigma ^n\alpha _0)D_0$ then $e$ can be written as 
$\sum _{a\in\Z ^0(p)}t^{-a}D_{a0}$, where $\Z ^0(p)=\Z ^+(p)\cup\{0\}$.

Fix $f\in G(\c L_{\c K_{sep}})$ such that $\sigma f=e\circ f$. 
Then for $\tau\in\c G $, 
the correspondence 
$$\tau\mapsto (-f)\circ \tau f\in G(\c L_{K_{sep}})
|_{\sigma =\op{id}}=G(\c L),$$
induces the identification of profinite groups 
$\eta _M:\c G_{<p}\simeq G(\c L)$. 

Note that $f\in \c L_{\c K_{<p}}$ and $\c G_{<p}$ strictly acts on the 
$\c G$-orbit of $f$. 
\medskip 

The above result is a covariant version of the nilpotent Artin-Schreier theory developed 
in \cite{Ab2}, cf. also Subsection 1.1 in \cite{Ab10} 
 for the relation between the covariant and contravariant 
versions of this theory and for appropriate non-formal comments. 
\medskip  

We shall use below a fixed choice of $f$ and use the notation for $e$ and $f$ without 
further references. 
\medskip 

\subsection{ Relation to class field theory}  \label{S1.5}

The above identification $\eta _M$ taken modulo 
$C_2(\c G_{<p})$ gives an isomorphism of profinite $p$-groups 
$$\eta _M^{ab}:\c G_{<p}^{\op{ab}}\longrightarrow  
\c L^{ab}=\c L/C_2(\c L)=\mathfrak M^{\Gal (k/\F _p)}=\c K^*/\c K^{*p^M}.$$

\begin{Prop} \label{P1.2} 
$\eta _M^{\,ab}$ is induced by the inverse to the reciprocity map 
of local class field theory $\kappa $. 
\end{Prop} 

\begin{proof} Indeed, let $\{\beta_i\}_{1\leqslant i\leqslant N_0}$ be a 
$\Z /p^M$-basis of $W_M(k)$ and let  
$\{\gamma _i\}_{1\leqslant i\leqslant N_0}$ be its dual 
basis with respect to the bilinear form 
induced by the trace of the field extension $W(k)[1/p]/\Q _p$. 

If $a\in\Z ^+(p)$ and $E(\beta _i,t_0^a)^{1/a}=D_{ia}$, then 
$D_{ia}=\sum _{n}\sigma ^n(\beta _i)D_{an}$, and, therefore, 
$D_{a0}=\sum_{i} \gamma _iD_{ia}$. This implies that  
$$e=\sum_{i,a} {t}^{-a}\gamma _iD_{ia}+\alpha _0D_0\,
\op{mod}\,C_2(\c L_{\c K}),$$
$$f=\sum_{ i,a} f_{ia}D_{ia}+f_0D_0\,
\op{mod}\,C_2(\c L_{\c K_{sep}}),$$
where all $f_{ia}, f_0\in O(\c K_{<p})$, $\sigma f_{ia}-f_{ia}=\gamma _i{t}^{-a}$ and   
$\sigma f_0-f_0=\alpha _0$. 
From the definition of $\eta _M$ it follows formally that 
for $\tau _{ia}=(\eta _M^{ab})^{-1}D_{ia}$ and $\tau _0=(\eta _M^{ab})^{-1}D_{0}$,    
$\tau _{ia}f_{i_1a_1}=f_{i_1a_1}+\delta (ii_1)\delta (aa_1)$,  
$\tau _0f_{i_1a_1}=f_{i_1a_1}$, $\tau _{ia}f_0=f_0$ and $\tau _0f_0=f_0+1$. 
(Here $\delta $ is the Kronecker symbol.)  

Now the explicit formula for the Hilbert symbol from Subsection \ref{S1.3}  
shows that  
$\kappa (E(\beta _i,t_0^a)^{1/a})$ and $\kappa (t_0)$  
act by the same formulae as $\tau _{ia}$ and, resp., $\tau _0$. 
\end{proof} 
\medskip 

\subsection{ Construction of lifts of analytic automorphisms } \label{S1.6}

Let $\eta _0\in\op{Aut}\c K$.  Then there is a lift  $\eta _{<p,0}
\in\op{Aut}\c K_{<p}$ of $\eta _0$.  (Use that the subgroup 
$\c G^{p^M}C_p(\c G)$ of 
$\c G$ is characteristic.)  For any another such lift $\eta '_{<p,0}$,  
we have $\eta '_{<p,0}\eta ^{-1}_{<p,0}\in\c G _{<p}$. 

The covariant version of the Witt-Artin-Schreier theory \cite{Ab2}, Section 1    
(cf. also \cite{Ab10}, Subsection 1.1 and 
\cite{Ab11}, Section 1),  
gives explicit description of 
the automorphisms $\eta _{<p,0}$ in terms of the 
identification $\eta _M$. 
Consider a special case of 
this construction when $\eta _0$ admits a lift 
$\eta\in\Aut\, O(\c K)$ which commutes with $\sigma $, and therefore 
we have the appropriate lifts 
$\eta _{<p}\in\op{Aut}\,O(\c K_{<p})$,  
cf. Subsection \ref{S1.1}.  
Then in terms of our fixed elements $e$ and $f$,  we have 
$\eta _{<p}(f)=c\circ (A\otimes \id _{O(\c K_{<p})})f$, 
where $c\in \c L_{\c K}$ and $A\in\Aut \c L$ can be found from the relation 
$$(\id _{\c L}\otimes \eta )e=\sigma c\,\circ (A\otimes\id _{O(\c K)})e\circ (-c)\, ,$$
cf. Subsection 1.5 in \cite{Ab2}, or Propositiobn 1.1 in \cite{Ab11}, 
and Subsection \ref{S3.2} below.

In other words, if $(A\otimes \id _{W_M(k)})(D_{a0})=\wt{D}_{a0}$ then  
$$\sum _{a\in\Z ^0(p)}\eta ({t})^{-a}D_{a0}=
\sigma c\,\circ 
\left (\sum _{a\in\Z ^0(p)}{t}^{-a}\wt{D}_{a0}\right )\circ (-c)\,.$$

Note that proceeding as in \cite{Ab2}, Subsection 1.5.4, cf. also \cite{Ab11}, 
Subsection 1.2, we can verify (this fact will be used systematically below) 
that with respect 
to the identification $\eta _M$, the automorphism $A$ coincides with the 
conjugation 
$\op{Ad}\,\eta _{<p}:\tau\mapsto \eta _{<p}^{-1}\,\tau\,\eta _{<p}$ (here 
$\tau\in\c G_{<p}$). 
\medskip 

\ \ 

\subsection{ Ramification filtration in $\c L$} \label{S1.7} 

For $v\geqslant 0$, denote by $\c G_{<p}^{(v)}$ the ramification subgroup 
of $\c G_{<p}$ with the upper index $v$. 
Let $\c L^{(v)}$ be the ideal of $\c L$ such that 
$\eta _M(\c G_{<p}^{(v)})=G(\c L^{(v)})$. The ideals 
$\c L^{(v)}$ have the following explicit description. 

First, for any $a\in\Z ^0(p)$ and $n\in\Z $, set $D_{an}:=D_{a,n\,\op{mod}\,N_0}$. 
In other words, we allow the second index 
in all $D_{an}$ to take integral values and assume that  
$D_{an_1}=D_{an_2}$ iff $n_1\equiv n_2\,\op{mod}\,N_0$.  
For $s\geqslant 1$, agree to use the notation 
$(\bar a,\bar n)_s$, where $\bar a=(a_1,\dots ,a_s)$ has coordinates 
in $\Z ^0(p)$ and $\bar n=(n_1,\dots ,n_s)\in\Z ^s$. Then we can attach 
to $(\bar a,\bar n)_s$ the commutator 
$[\dots [D_{a_1n_1},D_{a_2n_2}],\dots ,D_{a_sn_s}]$ and 
set $\gamma (\bar a, \bar n)_s=a_1p^{n_1}+\dots +a_sp^{n_s}$. 
For any $\gamma\geqslant 0$, let 
$\c F^0_{\gamma ,-N}$ be the element from $\c L_{k}$ given by 
\begin{equation} \label{E1.9} \c F^0_{\gamma ,-N}=
\sum  _{\gamma (\bar a,\bar n)_s=\gamma }p^{n_1}a_1\eta (\bar n)
[\dots [D_{a_1n_1},D_{a_2n_2}],\dots ,D_{a_sn_s}]
\end{equation} 
where $\eta (\bar n)$ equals $(s_1!(s_2-s_1)!\dots (s-s_{l})!)^{-1}$ if 
$0\leqslant n_1=\dots =n_{s_1}>n_{s_1+1}=
\dots =n_{s_2}>\dots >n_{s_{l}}=\dots =n_{s}\geqslant -N$, 
and equals to zero otherwise. Then the main result of \cite{Ab3} 
(translated into the covariant setting, cf. \cite{Ab4}, Subsections 1.1.2 and 1.2.4) states that 
\medskip 

$\bullet $\ {\it there is $\wt{N}(v)\in\N $ such that if we fix any 
$N\geqslant \wt{N}(v)$, then 
$\c L^{(v)}$ is the minimal ideal of $\c L$ such that for all 
$\gamma\geqslant v$,  
$\c F^0_{\gamma ,-N}\in\c L^{(v)}_k$. }
\medskip

\section {Filtration $\{\c L(s)\}_{s\ge 1}$} \label{S2}
 
In this section 
we define a decreasing central filtration 
$\{\c L(s)\}_{s\geqslant 1}$ in the $\Z /p^M$-Lie algebra 
$\c L$ from Subsection \ref{S1.4}. 
Its definition depends on a choice of a special element $S\in\m (\c K):=tW_M(k)[[t]]
\subset O(\c K)$. This element $S$ (together with the appropriate 
elements $S_0$ and $S'$ from its definition) will be specified in Section \ref{S4}, 
where we apply our results to the mixed characteristic case.

\subsection{Elements $S_0,S',S\in\m (\c K)$} \label{S2.1} \  
 
Let $[p]$ be the isogeny of multiplication by $p$ in the 
formal group $\op{Spf}\,\Z _p[[X]]$ over $\Z _p$ 
with the logarithm $X+X^p/p+\dots +X^{p^n}/p^n+\dots $.

Choose $S_0\in\m (\c K)$ and set $S'=[p]^{M-1}(S_0)$ and $S=[p]^M(S_0)$. 
Then $S,S'\in\m (\c K)$,  
they both depend only on the residue $S_0\,\op{mod}\,p$ and 
$S=\sigma S'$. In particular, if $e^*\in\N $ is such that 
$S\,\op{mod}\,p$ generates the ideal $(t_0^{e^*})$ in $k[[t_0]]$ then 
$e^*\equiv 0\,\op{mod}\,p^M$.  

\begin{Prop} \label {P2.1}
 {\rm a)} $dS=0$ in $\Omega ^1_{O(\c K)}$;
\medskip 

{\rm b)} there is $S^{\prime\prime }\in\m (\c K)$, such that 
$S=S'(p+S'')$; 
\medskip 

{\rm c)} there are $\eta _0,\eta _1\in W_M(k)[[t]]^\times $ 
and $\eta _2\in W_M(k)[[t]]$ such that  
$$S=t^{e^*}\eta _0+pt^{e^*/p}\eta _1+p^2\eta _2.$$ 
\end{Prop}
 
\begin{proof} a) The congruence $[p]X\equiv X^p\,\op{mod}\,p\Z _p[[X]]$ implies that 
$d([p]X)\in\,p\Z _p[[X]]$. Therefore, $dS=0$ in $\Omega ^1_{O(\c K)}$. 

b) Note that  $[p](X)\equiv pX\,\op{mod}\,X^2$. 
Therefore, there are $w_i\in\Z _p$ such that 
$S=[p]S'=pS'+\sum _{i\geqslant 2}w_iS^{\prime i}$ and we can take 
$S^{\prime\prime }=\sum _{i\geqslant 1}w_{i+1}S^{\prime i}$. 

c) The $t_0$-adic valuation of $S'\,\op{mod}\,p$ equals 
$e^*/p$. Then our property is implied 
by the following equivalence in $\Z _p[[X]]$
$$[p](X)\equiv pX+X^p\,\op{mod}\,(pX^{p^2-p+1}, p^2X).$$
\end{proof}

\begin{remark}
 We shall use below property a) in the following form: 
\medskip 

{\it if $s\in\N $ and $S^{s}=\sum_{ l\geqslant 1}\gamma _{ls}{t}^l$, 
where all $\gamma _{ls}\in W_M(k)$, then $l\gamma _{ls}=0$}.
\end{remark}
\medskip 

\subsection {Morphism $\iota $} \label{S2.2} \  
 Let $\c U=(1+t_0k[[t_0]])^{\times }$ be the $\Z _p$-module 
of principal units in $\c K$. 
Then $\c U/\c U^{p^M}$ is a closed $\Z /p^M$-submodule 
in $\c K^*/\c K^{*p^M}$. 
 Note that 
$\m (\c K)=W_M(\m _{\c K})\cap O(\c K)$, where $\m _{\c K}$ is the maximal ideal 
in the valuation ring of $\c K$.  
Consider a (unique) continuous homomorphism 
$$\iota :\c U\longrightarrow \m (\c K)$$
such that for any $\alpha\in W_M(k)$ and $a\in\Z ^+(p)$, 
$\iota : E(\alpha ,t_0^a)\mapsto\alpha {t}^a$ 
(here $E$ is the Shafarevich function, cf. Subsection \ref{S1.3}).  
 
Then $\iota $ induces an identification of $\c U/\c U^{p^M}$ with the closed 
$W_M(k)$-submodule  
$$\op{Im}\,\iota =
\left\{\sum _{a\in\Z ^+(p)}\alpha _at^a\ |\ \alpha _a\in W_M(k)\right\}$$ 
in $O(\c K)$. This submodule is topologically 
generated over $W_M(k)$ by all ${t}^a$ with $a\in\Z^+(p)$.
\medskip 

\subsection {Definition of $\{\c L(s)\}_{s\geqslant 1}$} \label{S2.3}

Set $(\c K^*/\c K^{*p^M})^{(1)}=\c K^*/\c K^{*p^M}$. For $s\geqslant 1$,  
let $(\c K^*/\c K^{*p^M})^{(s+1)}=(\op{Im}\,\iota )S^{s}$ 
 with respect to the identification 
$\c U/\c U^{p^M}=\op{Im}\,\iota $ from Subsection  \ref{S2.2}.  
Note, that $S=\sigma S'$ implies that for any $s\in\N $, 
$(\op{Im}\,\iota )S^s\subset \op{Im}\,\iota $.  

\begin{definition} $\{\c L(s)\}_{s\geqslant 1}$ is 
the minimal central filtration 
of ideals of the Lie algebra ${\c L}$ such that for all $s\geqslant 1$, 
$\c L(s)\supset (\c K^*/\c K^{*p^M})^{(s)}$.
\end{definition} 

The ideals $\c L(s)$ can be defined by induction on $s$ as follows. Let   
$\c L(1)={\c L}$; then for $s\geqslant 1$, the ideal 
$\c L(s+1)$ is generated by the elements of 
$(\c K^*/\c K^{*p^M})^{(s+1)}$ and 
$[\c L(s),{\c L}]$. 
Note also that for any $s$, $(\c K^*/\c K^{*p^M})
\cap \c L(s)=(\c K^*/\c K^{*p^M})^{(s)}$. 
(Use that $\Z /p^M$-module  $\c L(s)$ is isomorphic to 
$(\c K^*/\c K^{*p^M})^{(s)}\oplus (\c L(s)\cap C_2(\c L))$.
\medskip 

In addition, for any $s\geqslant 1$, the quotients 
$(\c K^*/\c K^{*p^M})^{(s)}/(\c K^*/\c K^{*p^M})^{(s+1)}$ are free $\Z/p^M$-modules. 
This easily implies that all $\c L(s)/\c L(s+1)$ are also free $\Z /p^M$-modules. 

\subsection {Characterization of  
$\{\c L(s)\}_{s\ge 1}$ in terms of $e\in\c L_{\c K}$} \label{S2.4}  
 Recall that  $e=\sum_{ a\in\Z^0(p)} {t}^{-a}D_{a0}$, 
cf. Subsection \ref{S1.4}.

\begin{Prop} \label{P2.2} 

The filtration $\{\c L(s)\}_{s\geqslant 1}$ is the minimal central filtration 
in ${\c L}$ such that $\c L(1)={\c L}$ and for all $s\geqslant 1$,   
$$S^{s}e\in {\c L}_{\m (\c K)}+
\c L(s+1)_{\c K}.$$ 
\end{Prop} 

\begin{proof} We need the following two lemmas. 
 
\begin{Lem} \label{L2.3}

For all $s\geqslant 1$ and $\alpha _a\in W_M(k)$ 
where $a\in\Z^+(p)$, we have  
$$\prod_{ a\in\Z^+(p)} E(\alpha _a,t_0^a)
\in (\c K^*/\c K^{*p^M})^{(s+1)}\Leftrightarrow 
\prod_{ a\in\Z^+(p)} E(\alpha _a,t_0^a)^{1/a}
\in (\c K^*/\c K^{*p^M})^{(s+1)}.$$
\end{Lem} 

\begin{proof}[Proof of Lemma]

We must prove that    
$$\sum_{ a\in\Z^+(p)} \alpha _a{t}^a\in 
S^{s}\m (\c K)\ \ \Leftrightarrow \ \ 
\sum_{ a\in\Z^+(p)} \frac{1}{a}\alpha _a
{t}^a\in S^{s}\m (\c K).$$

Let $S^{s}=\sum_{ l\geqslant 1}\gamma _{ls}{t}^l$ with  
$\gamma _{ls}\in W_M(k)$,  
then $l\gamma _{ls}=0$, cf. Remark in Subsection \ref{S2.1}. 

Suppose $\sum _{a\in\Z ^+(p)}\alpha _a{t}^{a}\in S^{s}\m (\c K)$. 

Then $\sum_a\alpha _a{t}^a=
(\sum_b\beta _b{t}^b)(\sum_l\gamma _{ls}{t}^l)$, 
where $\sum _b\beta _bt^b\in\m (\c K)$ 
 and  
$\alpha _a=\sum_{ a=b+l} \beta _b\gamma _{ls}$. This implies   
$$\frac{1}{a}\alpha _a=\sum_{ a=b+l} \frac{1}{a}\beta _b\gamma _{ls}=
\sum_{ a=b+l} \frac{1}{b}\beta _b\gamma _{ls},$$
because if $a=b+l$ and $a\in\Z^+(p)$ then  $b\in\Z^+(p)$ and 
$$\frac{1}{a}\gamma _{ls}-\frac{1}{b}\gamma _{ls}=
\frac{-l\gamma _{ls}}{ab}=0.$$

So, 
$$\sum_{ a\in\Z^+(p)} \frac{1}{a}\alpha _a{t}^a=
\left (\sum_{ b\in\Z^+(p)} \frac{1}{b}\beta _b{t}^b\right )
\left (\sum_{ l} \gamma _{ls}{t}^l\right )$$
and $\sum _{a}\frac{1}{a}\alpha _at^a\in S^{s}\m (\c K)$. 

Proceeding in the opposite direction we obtain 
the inverse statement. The lemma is proved.
\end{proof}  

\begin{Lem} \label{L2.4}

If $s\geqslant 1$ and all $\alpha _a\in W_M(k)$ then  
$$\prod_{ a\in\Z^+(p)}{\hskip -8pt} E(\alpha _a,t_0^a)^{1/a}
\in (\c K^*/\c K^{*p^M})^{(s)}\Leftrightarrow {\hskip -10pt}
\sum_{ a\in\Z^+(p)}{\hskip -7pt} \alpha _aD_{a0}\in  
(\c K^*/\c K^{*p^M})_k^{(s)}$$
\end{Lem} 
\begin{proof}[Proof of Lemma]

Suppose  
$$\prod_{ a\in\Z^+(p)} E(\alpha _a,t_0^a)^{1/a}
\in (\c K^*/\c K^{*p^M})^{(s)}.$$
Choose a $W_M(\F_p)$-basis $\{\beta _i\}$ of 
$W_M(k)$, and let $\{\gamma _i\}$ 
be its dual with respect to the trace form. Then for any $i$,  
$$\prod_{ a\in\Z^+(p)} E(\beta _i\alpha _a,t_0^a)^{1/a}
\in (\c K^*/\c K^{*p^M})^{(s)}.$$
In other words (use \eqref{E1.5} from Subsection \ref{S1.3}), 
$$c_i=\sum_{\substack{a\in\Z^+(p)\\n\in\Z/N_0\Z}} 
\sigma ^n(\beta _i)\sigma ^n(\alpha _a)D_{an}
\in \left (\c K^*/\c K^{*p^M}\right )^{(s)}\subset \c L(s),$$
and 
$$\sum_{ i} \gamma _ic_i=\sum_{ a\in\Z^+(p)} \alpha _a D_{a0}
\in \c L(s)_k.$$

Suppose now that \  $\sum_{ a\in\Z^+(p)} 
\alpha _aD_{a0}\in \c L(s)_k$. 
Then 
$$\sum_{ a\in\Z^+(p)} 
\alpha _aD_{a0}\in (\c K^*/\c K^{*p^M})_k^{(s)},$$
and, therefore, 
$$\sum_{ \substack{a\in\Z^+(p)\\n\in\Z/N_0\Z}} 
\sigma ^n(\alpha _a)D_{an}\in(\c K^*/\c K^{*p^M})^{(s)}.$$
This means, that 
$$\prod_{ a\in\Z^+(p)} E(\alpha _a,t_0^a)^{1/a}
\in (\c K^*/\c K^{*p^M})^{(s)}.$$
The lemma is proved. 
\end{proof} 

Now we can finish the proof of our proposition. If, as earlier,  
$S^{s}=\sum_{ l\geqslant 1} 
\gamma _{ls}{t}^l$ with $\gamma _{ls}\in W_M(k)$, then  
$(\op{Im}\,\iota )S^{s}$ is the $W_M(k)$-submodule in $\m (\c K)$ 
generated by the elements 
${t}^{a_1}S^{s}=\sum_{ l\geqslant 1} 
\gamma _{ls}{t}^{l+a_1}$, $a_1\in\Z^+(p)$. 
The above lemmas imply then that $\{\c L(s)\}_{s\ge 1}$ 
is the minimal central filtration 
in ${\c L}$ such that $\c L(1)=\c L$ and for all $a_1\in\Z^+(p),s\ge 1$, 
$$\sum_{ l\geqslant 1} \gamma _{ls}D_{a_1+l,0}
\in \c L(s+1)_k\,.$$

On the other hand, 
$$S^{s}e=\sum_{\substack{a\in\Z^0(p)\\l\geqslant 1}} 
\gamma _{ls}{t}^{-(a-l)}D_{a0}\equiv 
\sum_{ a_1\in\Z^+(p)} \left (\sum_{ l\geqslant 1} 
\gamma _{ls}
D_{a_1+l,0}\right ){t}^{-a_1}$$
modulo $\c L_{ \m (\c K)}$.
Therefore, 
$$S^{s}e\in {\c L}_{\m (\c K)}+\c L(s+1)_{\c K}
\Leftrightarrow $$
$$\sum_{l} \gamma _{ls}D_{a_1+l,0}\in \c L(s+1)_k\ \ 
\text{for all }a_1\in\Z^+(p).$$ 

The proposition is proved. 
\end{proof}

\begin{definition}  
 $\c N=\sum _{s\geqslant 1}{S^{-s}}\c L(s)_{\m (\c K)}$.
\end{definition}

Note that $\c N$ is a Lie $W_M(\F _p)$-subalgebra 
in $\c L_{\c K}$.  With this notation Proposition \ref{P2.2} 
implies the following characterization of 
the filtration $\{\c L(s)\}_{s\geqslant 1}$.

\begin{Cor} \label{C2.5} 
 $\{\c L(s)\}_{s\geqslant 1}$ is the minimal central filtration in $\c L$ 
such that $\c L(1)=\c L$ and $e\in\c N$. 
\end{Cor}

\begin{proof}
 It will be sufficient to verify that 
 $$e\in\c N\ \Leftrightarrow \ \forall s\geqslant 1,\ S^se\in\c L_{\m (\c K)}+\c L(s+1)_{\c K}\,.$$
 The ``if'' part is obvious. The ``only if'' part can be proved by induction on $s$ via the following property: 
 \medskip 
 
--- if $l'(s)\in \c L(s)_{\c K}$ and $Sl'(s)\in \c L_{\m (\c K)}+\c L(s+1)_{\c K}$ then 
 $l'(s)\in S^{-1}\c L(s)_{\m (\c K)}+\c L(s+1)_{\c K}$ 
 (use that $\c L(s)/\c L(s+1)$ is free $\Z /p^M$-module). 
\end{proof}

\medskip 

\subsection {Element $e^{\dag }\in G(\c L_{\c K})$ } \label{S2.5}  

Recall that  
$S\,\op{mod}\,p$ generates the ideal $(t_0^{e^*})$ in $k[[t_0]]$. 
 Therefore, 
the projections of the elements of the set 
$$\left\{ S^{-m}{t^b}\ |\ 1\leqslant b<e^*, 
\op{gcd}(b,p)=1, m\in\N \right\}\cup \{\alpha _0\}$$
form a basis of $O(\c K)/(\sigma -\id )O(\c K)$ over $W_M(k)$.

\begin{Prop} \label{P2.6}
There are $V_{(0)}\in\c L$, $x\in S\c N$ and $V_{(b,m)}\in 
\c L_{k}$ , where 
$m\geqslant 1$, $1\leqslant b<e^*$, $\op{gcd}(b,p)=1$, 
such that 
\medskip 
 
{\rm a)} $e^{\dag}:=\sum_{m,b} S^{-m}{t}^bV_{(b,m)}+
\alpha _0V_{(0)}\in\c N$;
\medskip 

{\rm b)} $e^{\dag}=(-\sigma x)\circ e\circ x$.
\end{Prop}

\begin{proof}

Note that $S\in\sigma \m (\c K)$ implies that the sets 
$\c\{t^{-a}\ |\  a\in\Z ^+(p)\}$ and 
$\{S^{-m}t^b\ |\ m\in\N , \op{gcd}(b,p)=1, 1\leqslant b<e^*\}$ 
generate the same $W_M(k)$-submodules in $O(\c K)/\m (\c K)$. 
This implies the existence of $V_{(0)}^{(0)}\in\c L$ and 
$V_{(b,m)}^{(0)}\in\c L_k$ such that 
\begin{equation} \label {E2.1} e\equiv e_0^{\dag }\,\op{mod}\,\c L _{\m (\c K)}
 \end{equation} 
where $e_0^{\dag }:=\sum _{(b,m)}S^{-m}t^bV^{(0)}_{(b,m)}+
\alpha _0V_{(0)}^{(0)}$.

For $i\geqslant 1$, let $\c N^{(i)}=
\sum _{s\geqslant i}S^{-s}\c L(s)_{\m (\c K)}$. Then 
\medskip 
 
$\bullet $\ $\c N^{(i)}=S^{-i}\c L(i)_{\m (\c K)}+\c N^{(i+1)}$; 
\medskip 
 
$\bullet $\ $\left [\c N^{(i)}, \c N\right ]
\subset \c N^{(i+1)}$. 
\medskip 

In particular, relation \eqref{E2.1} implies that $e=e_0^{\dag }+
\sigma x_0-x_0$, where $x_0\in \c L_{\m (\c K)}$, and we obtain 
\begin{equation} \label{E2.2} 
 (-\sigma x_0)\circ e\circ x_0\equiv e_0^{\dag}\,\op{mod}\,S\c N^{(2)}
\end{equation}
(use that $x_0,\sigma x_0\in \c L_{\m (\c K)}\subset S\c N^{(1)}$). 
Now we need the following lemma.

\begin{Lem} \label{L2.7} 
Suppose $\mathfrak M$ is a $\Z _p$-module and $i_0\in\N $. 
Then for any $l\in S^{-i_0}\mathfrak M_{\m (\c K)}$, there are 
$l_{(0)}\in\mathfrak M$, $\tilde l\in S^{-i_0}\mathfrak M_{\m (\c K)}$ and 
$l_{(b,m)}\in\mathfrak M_k$, where 
$1\leqslant m\leqslant i_0$, $\op{gcd}(p,b)=1$ and  
$1\leqslant b<e^*$, such that 
$$l=\sum _{b,m}S^{-m}t^bl_{(b,m)}+\alpha _0l_{(0)}+
\sigma\tilde l-\tilde l.$$
 \end{Lem}

\begin{proof}[Proof of Lemma \ref{L2.7}]
It will be sufficient to consider the case $\mathfrak M=\Z _p$. 
In other words, we must prove the following statement: 
\medskip 
 
$\bullet $\ {\it For any $s\in S^{-i_0}\m (\c K)$, there are 
$\beta _{(0)}\in W_M(\F _p)$, 
$\tilde s\in S^{-i_0}\m (\c K)$ and $\beta _{(b,m)}\in W_M(k)$, 
where $1\leqslant m\leqslant i_0$, $\op{gcd}(b,p)=1$ and $1\leqslant b<e^*$, 
such that} 
$$s=\sum _{b,m}\beta _{(b,m)}S^{-m}t^b+
\alpha _0\beta _{(0)}+\sigma\tilde s-\tilde s.$$
\medskip 

 We can assume that  
$s=t^{a_0}/S^{i_0}$, where 
$1\leqslant a_0<e^*$, $i_0\in\N $ and our lemma is proved 
for all elements $s$ from $pS^{-i_0}\m (\c K)+t^{a_0}S^{-i_0}\m (\c K)$. 

If $\op{gcd}\,(a_0,p)=1$ there is nothing to prove. Otherwise, 
$a_0=pa_1$ and $s=s'+\sigma (s')-s'$ with 
$s'=t^{a_1}/S^{\prime i_0}=t^{a_1}(p+S^{\prime\prime })/S^{i_0}$. 
It remains to note that $s'\in pS^{-i_0}\m (\c K)+t^{a_0}S^{-i_0}\m (\c K)$, 
because $S^{\prime\prime }\,\op{mod}\,p\in (t_0^{e^0})$, where $e^0:=e^*(1-1/p)$,  and 
$a_1+e^0=a_0/p+e^0>a_0$ (use that $a_0<e^*$).  
\end{proof} 

Continue the proof of Proposition \ref{P2.6}. 
Clearly, it is implied by the following lemma. 

\begin{Lem} \label{L2.8} 
 
For all $i\geqslant 0$, there are 
$x_i\in S\c N$, $V^{(i)}_{(b,m)}\in\c L_{k}$ 
and $V^{(i)}_{(0)}\in\c L$ such that:
\medskip 

{\rm $a_1$)} $x_{i+1}\equiv x_{i}\,\op{mod}\,S\c N^{(i+1)}$;
\medskip 

{\rm $a_2$)} $V^{(i+1)}_{(b,m)}\equiv 
V^{(i)}_{(b,m)}\,\op{mod}\,\c L(i+2)_k$; 
\medskip 

{\rm $a_3$)} $V^{(i+1)}_{(0)}\equiv V^{(i)}_{(0)}\,\op{mod}\,\c L(i+2)$
\medskip 

{\rm b)} if $e_i^{\dag }=\sum _{b,m}S^{-m}t^bV_{(b,m)}^{(i)}+
\alpha _0V_0^{(i)}$ then 
$$(-\sigma x_i)\circ e\circ x_i\equiv e_i^{\dag }
\,\op{mod}\,S\c N^{(i+2)}.$$
\end{Lem}

\begin{proof} Use the elements $V_{(b,m)}^{(0)}, V_{(0)}^{(0)}, 
e_0^{\dag }$ and $x_0$ from the beginning 
of the proof of Proposition \ref{P2.6}.  Then part b) 
holds for $i=0$ by \eqref{E2.2}. 

Let $i_0\geqslant 1$ and assume that our 
Lemma is proved for all $i<i_0$. 
Let $l\in S^{-i_0}\c L(i_0+1)_{\m (\c K)}$ be such that  
$$e^{\dag} _{i_0-1}-(-\sigma x_{i_0-1})\circ e\circ x_{i_0-1}
\equiv l\,\op{mod}\,S\c N^{(i_0+2)}\,.$$
Apply Lemma \ref{L2.7} to $\mathfrak M=\c L(i_0+1)$ and 
$l\in S^{-i_0}\c L(i_0+1)_{\m (\c K)}$. This gives us the 
appropriate elements $l_{(b,m)}\in \c L(i_0+1)_k$, 
$l_{(0)}\in \c L(i_0+1)$ and 
$\tilde l\in S^{-i_0}\c L(i_0+1)_{\m (\c K)}$. Note 
that the elements $l_{(b,m)}$ are defined 
only for $1\leqslant m\leqslant i_0$. Extend their 
definition by setting $l_{(b,m)}=0$ if $m>i_0$. 
Then the case $i=i_0$ of Lemma \ref{L2.8} holds with 
$V_{(b,m)}^{(i_0)}=V_{(b,m)}^{(i_0-1)}+l_{(b,m)}$, 
$V_{(0)}^{(i_0)}=V_{(0)}^{(i_0-1)}+l_{(0)}$ and 
$x_{i_0}=x_{i_0-1}+\tilde l$. (We use here that  
$S\c N^{(i_0+1)}=S^{-i_0}\c L(i_0+1)_{\m (\c K)}+S\c N^{(i_0+2)}$.)   

Lemma \ref{L2.8} and Proposition \ref{P2.6} are completely proved. 
\end{proof}

\end{proof}

 Proposition \ref{P2.6}b) implies that 
the elements $\sigma ^nV_{(b,m)}$, $n\in\Z /N_0$,  together 
with $V_{(0)}$ form a system of free topological generators of $\c L_k$.
 Suppose $\{\beta _i\}_{1\leqslant i\leqslant N_0}$ and 
$\{\gamma _i\}_{1\leqslant i\leqslant N_0}$ are the 
$\Z /p^M$-bases of $W_M(k)$ from the proof of Proposition \ref{P1.2}. 
Proceeding similarly to that proof introduce the elements 
$$V_{(b,m),i}:=\sum _{n\in\Z /N_0}\sigma ^n(\beta _i)\sigma ^n(V_{(b,m)})\,.$$
Then all $V_{(b,m)}$ can be recovered via the relation 
$V_{(b,m)}=\sum _i\gamma _iV_{(b,m),i}$. This implies that  
the elements 
$V_{(b,m),i}$ together with $V_{(0)}$ form a system of 
free topological generators of $\c L$. 
(Recall that $\c L$ is a free object in 
the category of Lie $\Z /p^M$-algebras of nilpotent class $<p$.) Therefore, 
we can introduce the weight function $\op{wt}$ on 
$\c L$ by setting for all $b,m,i$, 
$\op{wt}(V_{(b,m),i})=m$ and $\op{wt}(V_{(0)})=1$. Note that by  
Proposition \ref{P2.6}b) we have that
$e^{\dag}\in\c N$ if and only if 
$e\in\c N$. Now   
Proposition \ref{P2.2} implies the following corollary. 

\begin{Cor} \label{C2.7} 
For any $s\geqslant 1$, $\c L(s)=\{l\in\c L\ |\ \op{wt}(l)\geqslant s\}$. 
\end{Cor}

\section {The groups $\wt{\c G}_{h}$ and $\c G_h$} \label{S3}

\subsection {Automorphism $h$} \label{S3.1}

Let $S\in O(\c K)$ be the element introduced in Subsection \ref{S2.1}.   
Let $h_0\in\Aut (\c K)$ be such that $h_0|_k=\id $ and 
$h_0(t_0)=t_0E(1,S\,\op{mod}\,p)$.  
Then $h_0$ admits a lift to $h\in\Aut\, O(\c K)$ 
such that $h|_{W_M(k)}=\id $ and  
$h(t)=tE(1,S)$. Recall that $O(\c K)=W_M(k)((t))$. 
If $n\in\N $ then denote by $h^n(t)$ the $n$-th superposition of 
the formal power series $h(t)$.  

\begin{Prop} \label{P3.1} 
 For any $n\in\N $, $h^n(t)\equiv tE(n,S)\,\op{mod}\,S^p\m (\c K)$
\end{Prop}

\begin{proof}
 If $n=1$ there is nothing to prove. 
Suppose proposition is proved for some $n\in\N $. Then 
$$h^{n+1}(t)=h^n(h(t))\equiv tE(1,S)E(n,S(h(t)))
\,\op{mod}\,\m (\c K)S(h(t))^p.$$ 
Recall, cf. Subsection \ref{S2.2},  that $S=\sum _{l\geqslant 1}
\gamma _{l1}t^l$, where 
$\gamma _{l1}\in W_M(k)$ and 
$\gamma _{l1}l=0$. Let $l=l'p^a$ with $\op{gcd}(l',p)=1$. 
Then $\gamma _{l1}\in p^{M-a}W_M(k)$.   

With the above notation we have in $W_M(k)[[t]]$, 
$$E(1,S)^l=\exp (p^aS+\dots +pS^{p^{a-1}})^{l'}E(1,S^{p^a})^{l'}
\equiv 1\,\op{mod}\,(p^a, S^p).$$
Therefore (use that $\gamma _{l1}p^a=0$), 
$$ 
S(h(t))\equiv S(tE(1,S))\equiv \sum _l\gamma _{l1}t^lE(1,S)^l
\equiv \sum _l\gamma _{l1}t^l=S\,\op{mod}\,S^p\,,$$ 
and $h^{n+1}(t)\equiv tE(1,S)E(n,S)
\equiv tE(n+1,S)\,\op{mod}\,\m (\c K)S^p$ (use that $S(h(t))^p\equiv 0\,\op{mod}\,S^p$).  
\end{proof}

\subsection{Specification of lifts $h_{<p}$} \label{S3.2} 

Note that $h(t)=t\alpha ^{p^{M-1}}$, where $\alpha =E(1,S_0)^p$, and therefore, 
$h$ commutes with $\sigma $, cf. Remark in Subsection \ref{S1.1}. 
Now suppose that 
$h_{<p,0}\in\op{Aut}\,\c K_{<p}$ is a lift of $h_0$. 
Then Proposition \ref{P1.1} provides us with 
a unique $h _{<p}\in\op{Aut}\,O(\c K_{<p})$ such that 
$h_{<p}|_{O(\c K)}=h$ and $h _{<p}\,\op{mod}\,p=h_{<p,0}$. Therefore, 
we can work with arbitrary lifts $h_{<p,0}$ of $h_0$ by working with 
the appropriate lifts $h_{<p}$ of $h$. Note that all such lifts 
$h_{<p}$ commute with $\sigma $. 
 
A lift $h _{<p}$  
of $h$ can be specified by the formalism 
of nilpotent Artin-Schreier theory 
as follows.

Define similarly to \cite{Ab11} the continuous $W_M(k)$-linear operators 
$\c R,\c S:\c L_{\c K}\To \c L_{\c K}$ as follows. 

Suppose $\alpha\in\c L_{k}$. 

For $n>0$, set $\c R(t^n\alpha )=0$ and $\c S(t^n\alpha )=
-\sum _{i\geqslant 0}\sigma ^i(t^n\alpha )$.

For $n=0$, set $\c R(\alpha )=\alpha _0(\id _{\c L}\otimes\Tr)(\alpha )$, 
\ \ $\c S(\alpha )=\sum _{0\leqslant j<i<N_0}
\sigma ^j\alpha _0\,\sigma ^i\alpha $, where 
$\op{Tr}:W_M(k)\To W_M(k)$ is induced by 
the trace map in $k/\F _p$ and $\alpha _0\in W_M(k)$ with  
$\op{Tr}\alpha _0=1$ was fixed in Subsection \ref{S1.4}. 

For $n=-n_1p^m$, $\op{gcd}(n_1,p)=1$, set    
$\c R(t^n\alpha )=t^{-n_1}\sigma ^{-m}\alpha $ and 
$\c S(t^n\alpha )=\sum _{1\leqslant i\leqslant m}\sigma ^{-i}(t^n\alpha )$. 

Similarly to \cite{Ab11} we have the following lemma. 
(We use also the special case $\mathfrak M=\Z _p$ of 
Lemma \ref{L2.7}.)

\begin{Lem} \label{L3.2} 
For any $b\in\c L_{\c K}$,  

{\rm a)}\ $b=\c R(b)+(\sigma -\id _{\c L_{\c K}})\c S(b)$;
\medskip 

{\rm b)}\ if $b=b_1+\sigma c-c$, where 
$b_1\in\sum _{a\in\Z ^+(p)}t^{-a}\c L_{k}+\alpha _0\c L$ 
and $c\in\c L_{\c K}$ then $\c R(b)=b_1$ and $c-\c S(b)\in \c L$; 
\medskip 

{\rm c)}\ for any $n\geqslant 0$, $\c R$ and $\c S$ 
map $S^{-n}\c L_{\m (\c K)}$ to itself.  
\end{Lem}

According to Subsection \ref{S1.6}, for the lift  
$h_{<p}\in\Aut\,O(\c K_{<p})$ of $h$ (which is attached to 
the lift $h _{<p,0}$ of $h_0$), we have  that    
$$h _{<p}(f)=c\circ (A\otimes\id _{O(\c K_{<p})})f\,.$$  
Here $c\in \c L_{\c K}$ and $A=\Ad\, h _{<p}\in\Aut\,\c L$ 
(cf. Subsection \ref{S1.6} for the definition of $\Ad\,h_{<p}$).  
Similarly to \cite{Ab11} it can be proved that the correspondence 
$h_{<p}\mapsto (c,A)$ is a bijection between the set of all lifts $h_{<p}$ of $h$ and all 
$(c,A)\in\c L_{\c K}\times\Aut\,\c L$ such that 
\begin{equation} \label{E3.1}
 (\id _{\c L}\otimes h)(e)\circ c=(\sigma c)\circ (A\otimes\id _{O(\c K}))(e)\,.
\end{equation} 
This allows us to specify a choice of $h _{<p}$ step by step 
proceeding from $h _{<p}\,\op{mod}\,C_s(\c L_{\c K_{<p}})$ to 
$h _{<p}\,\op{mod}\,C_{s+1}(\c L_{\c K_{<p}})$ 
where $1\leqslant s<p$, as follows.

Suppose $c$ and $A$ are already 
chosen modulo $s$-th commutators, i.e. we chose $(c_s,A_s)\in\c L_{\c K}\times\Aut\,\c L$ 
satisfying the relation \eqref{E3.1} modulo $C_s(\c L_{\c K})$. 

Then set $c_{s+1}=c_s+X$ and $A_{s+1}=A_s+\c A$, where 
$X\in C_s(\c L_{\c K})$ and $\c A\in\Hom (\c L, C_s(\c L))$.  
Then \eqref{E3.1} implies that (here $\c A_k=\c A\otimes W_M(k)$)
$$\sigma X-X+\sum _{a\in\Z ^0(p)}t^{-a}\c A_k(D_{a0})\equiv $$
\begin{equation}\label{E3.2} (\id _{\c L}\otimes h)e
\circ c_s-\sigma c_s\circ (A_s\otimes\id _{O(\c K)})
e\,\op{mod}\,C_{s+1}(\c L_{\c K})
\end{equation}

Now we can specify $c_{s+1}$ and $A_{s+1}$ by setting 
$X=\c S(B_s)$ and 
\linebreak 
$\sum _{a\in\Z ^0(p)}t^{-a}\c A_k(D_{a0})
=\c R(B_s)$, 
where $B_s$ is the right-hand side of the above recurrent    
relation. Note that the knowledge of all $\c A_k(D_{a0})$ recovers uniquely 
the values of $\c A$ on generators of $\c L$ and gives  well-defined 
$A_{s+1}\in\Aut\,\c L$. Clearly, $(c_{s+1}, A_{s+1})$ satisfies the relation 
\eqref{E3.1} modulo $C_{s+1}(\c L_{\c K})$. Finally, we obtain 
the solution $(c^0,A^0):=(c_p,A_p)$ of \eqref{E3.1} and can use it to 
specify uniquely the lift $h^0_{<p}$ of $h$. 

\subsection{The group $\wt{\c G}_{h}$} \label{S3.3} 
 
Consider the group of 
all continuous automorphisms of 
$\c K_{<p}$ such that 
their restriction to $\c K$ belongs to 
the closed subgroup 
in $\Aut\,\c K$ generated by $h _0$. These automorphisms 
admit unique lifts to automorphisms of $O(\c K_{<p})$ such that 
their restriction to $O(\c K)$ belongs to the subgroup $\langle h\rangle $ 
of $\Aut\,O(\c K)$ generated by $h$, cf. the beginning of 
Subsection \ref{S3.2}.  Denote the group of these 
lifts by  $\wt{\c G}_h$. 

Use the identification $\eta _M$ from Subsection \ref{S1.4} to 
obtain a natural 
short exact sequence of profinite $p$-groups 
\begin{equation} \label{E3.3}
1\To G(\c L)\To \wt{\c G}_{h}\To \langle h\rangle\To 1 
\end{equation} 
For any $s\geqslant 2$, the $s$-th commutator subgroup 
$C_s(\wt{\c G}_{h})$ is a normal subgroup in 
$G(\c L)$. Therefore, $\c L_h(s):=C_s(\wt{\c G}_{h})$ 
is a Lie subalgebra of $\c L$. 
Set $\c L_h(1)=\c L$. 
Clearly, for any $s_1,s_2\geqslant 1$,  
$[\c L_h(s_1),\c L_h(s_2)]\subset \c L_h(s_1+s_2)$, in other words, the filtration 
$\{\c L_h(s)\}_{s\geqslant 1}$ is  central.

\begin{Thm} \label{T3.3} 
 For all $s\in\N $, $\c L_h(s)=\c L(s)$. 
\end{Thm}

\begin{proof} Use the notation from Subsection \ref{S2.5}. 
Obviously, we have:  
\medskip 

$\bullet $\ $\c L(s+1)=\left (\c K^*/\c K^{*p^M}\right )^{(s+1)}
+\c L(s+1)\cap C_2(\c L)$, 
where the $W_M(k)$-module $\left (\c K^*/\c K^{*p^M}\right )^{(s+1)}$ 
is generated by all $V_{(b,m)}$ with $m\geqslant s+1$ 
(for the definition of $V_{(b.m)}$ cf. Proposition \ref{P2.6}) and 
$\c L(s+1)\cap C_2(\c L)=\sum _{s_1+s_2=s+1} \left [\c L(s_1),\c L(s_2)\right ]$;
\medskip 

$\bullet $\ $\c L_h(s+1)$ is the ideal in $\c L$ generated by  
$[\c L_h(s),\c L]$ and all elements of the form 
$(\op{Ad}h_{<p})l\circ (-l)$, where $l\in \c L_h(s)$ and $h_{<p}$ is a lift of $h$. 
\medskip

Consider the elements $V_{(0)}$ and  $V_{(b,m),i}$ introduced in the end of 
Section \ref{S2}). Recall that $m\in\N $, $1\leqslant b<e^*$ and 
$\op{gcd}(b,p)=1$. 

\begin{Lem} \label{L3.4} 
 There is a lift $h^1_{<p}$ such that if  
 $(\op{Ad}h^1_{<p})V_{(0)}=\wt{V}_{(0)}$ and for all $b,m,i$, 
$(\op{Ad}h^1_{<p})V_{(b,m),i}=\wt{V}_{(b,m),i}$ then 
\medskip 

{\rm a)}\ $\wt{V}_{(0)}\equiv V_{(0)}\,\op{mod}\ C_2(\c L)$;
\medskip  
 
{\rm b)}\ $\wt{V}_{(b,m),i}\equiv V_{(b,m),i}+bV_{(b,m+1),i}
\,\op{mod}\,(\c L(m+2)+\c L(m+1)\cap C_2(\c L))$. 
\end{Lem} 

We shall prove this Lemma below. 

Note the following immediate applications of this lemma: 
\medskip 

(a)\ {\it if $l\in \c L(s)$ then $(\op{Ad}h^1_{<p})l\circ (-l)\in \c L(s+1)$};
\medskip 

(b)\ {\it if $l\in \left (\c K^*/\c K^{*p^M}\right )^{(s+1)}$ then there is an 
$l'\in \left (\c K^*/\c K^{*p^M}\right )^{(s)}$ such that} 
$(\op{Ad}h^1_{<p})l'\circ (-l')\equiv l\,\op{mod}\,\c L(s+1)\cap C_2(\c L)$. 
\medskip 

Now we can finish the proof of our theorem. 

Clearly, $\c L_h(1)=\c L(1)$. 

Suppose $s_0\geqslant 1$ and for $1\leqslant s\leqslant s_0$, we have 
$\c L_h(s)=\c L(s)$. 

Then $[\c L_h(s_0),\c L]=[\c L(s_0), \c L(1)]\subset \c L(s_0+1)$ 
and applying (a) we obtain that $\c L_h(s_0+1)\subset \c L(s_0+1)$. 

In the opposite direction, note that by inductive assumption, 

$$\c L(s_0+1)\cap C_2(\c L)=\sum _{s_1+s_2=s_0+1}
\left [\c L_h(s_1),\c L_h(s_2)\right ]\subset \c L_h(s_0+1)$$
and then from (b) we obtain that $\left (\c K^*/\c K^{*p^M}\right )^{(s_0+1)}
\subset \c L_h(s_0+1)$. So, $\c L(s_0+1)\subset\c L_h(s_0+1)$.  
The theorem is completely proved. 
\end{proof}

\begin{proof} [Proof of Lemma \ref{L3.4}] 

Let 
$$\tilde e^{\dag}:=(\op{Ad}h^1 _{<p} \otimes\id _{O(\c K)})e^{\dag }
=\sum _{i,b,m}\frac{t^b}{S^m}\beta _i
\wt{V}_{(b,m),i}+\alpha _{(0)}\wt{V}_{(0)}\, .$$
Similarly to  Subsection \ref{S3.2} there is $c^1\in\c L_{\c K}$ such that  
\begin{equation} \label{E3d}
(\id _{\c L}\otimes h)e^{\dag }\circ c^1=(\sigma c^1)\circ \tilde{e}^{\dag }\,.
\end{equation} 
and the choice of $h^1_{<p}$ can be specified by an analog of the recurrent 
procedure from the end of Subsection \ref{S3.2}. 

Namely, set $c^1_1=0$ and $A^1_1=\id _{\c L}$. 
Then for $1\leqslant s< p$, $(c_{s+1}^1,A^1_{s+1})$ can be defined as follows:
\medskip 

$\bullet $\ $B_s=(\id _{\c L}\otimes h)e^{\dag }\circ c_s^1-
(\sigma c^1_s)\circ (A_s^1\otimes\id _{\c K})e^{\dag }$
\medskip 

$\bullet $\ $X_s=\c S(B_s)$, $(\c A_s\otimes\id _{\c K})e^{\dag }=\c R(B_s)$;
\medskip 

$\bullet $\    $c^1_{s+1}=c^1_s+X_s$, $A^1_{s+1}=A^1_s+\c A_s$
\medskip 

This gives  
the system of compatible on $1\leqslant s\leqslant p$ solutions 
$(c^1_s, A^1_s)\in\c L_{\c K}\times\Aut\,\c L$ of \eqref{E3d} 
modulo $C_s(\c L_{\c K})$ and $(c^1,A^1):=(c_p^1, A^1_p)$  
defines $h^1_{<p}$.

Let 
$$\wt{\c N}^{(2)}:=\sum _{i\geqslant 2}S^{-i}(\c L(i)\cap C_2(\c L))_{\m (\c K)}
\subset \c N^{(2)}\,.$$ 
Note that   
$\left [\c N, \c N\right ]\subset \wt{\c N}^{(2)}$. Consider the following properties. 
\medskip 

a) $(\id _{\c L}\otimes h)(e^{\dag })=e^{\dag }+e_1^++e_1^-\,\op{mod}\,S^2\c N$, where 
$e_1^+,e_1^-\in S\c N$ and 
$$e_1^-=\sum _{i,b,m}\frac{bt^b}{S^m}\beta _iV_{(b,m+1),i},\ \ e_1^+=\sum _{b,i}bt^b\beta _iV_{(b,1),i}$$
(use that $h(S)\equiv S(h(t))\equiv S\,\op{mod}\,S^p$, cf. the proof of Proposition \ref{P3.1}).
\medskip 

b) $\tilde{e}^{\dag }\equiv e^{\dag }\,\op{mod}\,S\c N$ 
and $c^1\in S\c N$ (use that for all $s$, $B_s\in S\c N$ and $\c R$ and $\c S$ map $S\c N$ to itself). 
\medskip 

c) $(-\sigma c^1)\circ (\id _{\c L}\otimes h)(e^{\dag })\circ c^1
\equiv (c^1-\sigma c^1)+e^{\dag }+e_1^{\dag }\,\op{mod}\,
S^2\c N+S\wt{\c N}^{(2)}$ 
(use that $c\in S\c N$ and $(\id _{\c L}\otimes h)(e^{\dag })\in \c N$)
\medskip 

d) Apply $\c R$ to the congruence from c), use that $S^2\c N+S\wt{\c N}^{(2)}$ is mapped by $\c R$ to itself and  
$\c R(c^1-\sigma c^1)=\c R(e_1^+)=0$
 
$$\tilde{e}^{\dag }\equiv \sum _{i,b,m}\frac{t^b}{S^m}\beta _i\left (V_{(b,m),i}+
bV_{(b,m+1),i}\right )+\alpha _0V_{(0)}\,
\op{mod}\,S^2\c N+S\wt{\c N}^{(2)}\,.$$
\medskip 

It remains to note that the last congruence is equivalent to the statement of our lemma.
\end{proof}

\subsection{The group $\c G_{h}$} \label{S3.4} 

Let $\c G_{h}=\wt{\c G}_{h}/\wt{\c G}^{p^M}_{h}C_p(\wt{\c G}_{h})$. 

\begin{Prop} \label{P3.5} 
Exact sequence \eqref{E3.3} induces the following 
exact sequence of $p$-groups 
\begin{equation} \label{E3.5} 
1\To G(\c L)/G(\c L(p))\To \c G_{h}\To \langle h\rangle \,\op{mod}\,
\langle h^{p^M}\rangle \To 1
\end{equation}
 \end{Prop}
 
 \begin{proof}  Set 

 $$\c M:=\c N+\c L(p)_{\c K}=\sum _{1\leqslant s<p}
 S^{-s}\c L(s)_{\m (\c K)}+\c L(p)_{\c K}$$

 $$\c M_{<p}:=\sum _{1\leqslant s<p}
 S^{-s}\c L(s)_{\m (\c K_{<p})}+\c L(p)_{\c K_{<p}}$$  
 where $\m (\c K_{<p})=W_M(\m _{<p})
\cap O(\c K_{<p})$ and $\m _{<p}$  is the maximal ideal 
of the valuation ring of $\c K_{<p}$.  
 
 Then $\c M$ has the induced structure of Lie 
$W_M(k)$-algebra (use the Lie bracket from $\c L_{\c K}$) 
and $S^{p-1}\c M$ is an ideal 
 in $\c M$. Similarly, 
 $\c M_{<p}$ is a Lie $W_M(k)$-algebra (containing $\c M$ as its subalgebra)  
 and $S^{p-1}\c M_{<p}$ is an  ideal in $\c M_{<p}$. 
 Note that $e\in\c M$, $f\in\c M_{<p}$, 
 $S^{p-1}\c M_{<p}\cap \c M=S^{p-1}\c M$, and we have 
 a natural embedding of $\bar{\c M}:=
\c M/S^{p-1}\c M$ 
 into $\bar{\c M}_{<p}:=
\c M_{<p}/S^{p-1}\c M_{<p}$. 
For $i\geqslant 0$, we have also  
$(\id _{\c L}\otimes h-\id _{\c M})^i\c M\subset S^i\c M$.  
 
 Consider the orbit of $\bar f:=f\,\op{mod}
 \,S^{p-1}\c M_{<p}$ with 
 respect to the natural action of $\wt{\c G}_{h}
\subset \Aut \,O(\c K_{<p})$ 
 on $\bar{\c M}_{<p}$. Prove that the stabilizer $\c H$ of $\bar f$ equals 
 $\wt{\c G}_{h}^{p^M}C_p(\wt{\c G}_{h})$.

 If $l\in G(\c L)$ then $\eta _M^{-1}(l)\in\c G_{<p}$ sends $f$ to $f\circ l$.   
This means that for  
$l\in\c L\cap\c H$ we have  
$$l\in S^{p-1}\c M_{<p}\cap \c L=S^{p-1}\c M\cap \c L
=\c L(p)_{\c K}\cap \c L=\c L(p)= C_p(\wt{\c G}_h)\, .$$ 
Therefore, 
$\c H\cap G(\c L)=C_p(\wt{\c G}_{h})\subset \c H$ and we obtain the embedding 
$$\kappa : G(\c L)/G(\c L(p))\longrightarrow \wt{\c G}_h/\c H\, .$$  

Now consider the lift $h^0_{<p}$ from the end of Subsection \ref{S3.2}.

Note that $\wt{\c G}_{h}^{p^M}\,\op{mod}\,C_p(\wt{\c G}_{h})$ 
 is generated by $h _{<p}^{0p^M}$. Indeed, any finite $p$-group of nilpotent class $<p$ 
is $P$-regular, cf. \cite{Ha} Subsection 12.3. In particular, 
for any $g\in G(\c L)$, 
$(h^0_{<p}\circ g)^{p^M}\equiv h^{0p^M}_{<p}\circ g'\,\op{mod}\,C_p(\wt{\c G}_h)\,,$
where $g'$ is the product of $p^M$-th powers of elements 
from $G(\c L)$, but $G(\c L)$ has period $p^M$. 

As earlier, $h^0_{<p}f=c^0\circ (A^0\otimes\id _{\c K})f$. 
Note that $c^0\in S\c M$ (proceed similarly to the proof of Lemma \ref{L3.4}, step b)). 
 
 Then  \ \  
 $h_{<p}^{0p^M}(f)=$ 
$$(\id\otimes h) ^{p^M-1}\left (c^0\circ (A^0\otimes h^{-1})c^0\circ 
 \dots \circ (A^0\otimes h^{-1})^{p^M-1}c^0\right )\circ (A^{0p^M}\otimes\id )f\, .$$
 Clearly, $(A^0-\id _{\c L})^{p}\c L\subset\c L(p)$ and, therefore, 
 $(A^{0p^M}\otimes\id )\bar f=\bar f$. 

Similarly, $B=A^0\otimes h^{-1}$ is an automorphism of the Lie algebra $\c M$, 
and for all $s\geqslant 0$, 
 $(B-\id _{\c M})(S^s\c M)\subset 
 S^{s+1}\c M$.

\begin{Lem} \label{L3.6} 
 For any $m\in\c S\c M$, $m\circ B(m)\circ\dots \circ B^{p^M-1}m\in\c S^p\c M$. 
\end{Lem}

\begin{proof} Consider the Lie algebra $\mathfrak{M}=S\c M/S^p\c M$ with 
the filtration $\{\mathfrak{M}(i)\}_{i\geqslant 1}$ induced by the filtration 
$\{S^i\c M\}_{i\geqslant 1}$. 
This filtration is central, i.e. for any $i,j\geqslant 1$, 
$[\mathfrak{M}(i), \mathfrak{M}(j)]\subset\mathfrak{M}(i+j)$. 
In particular, the nilpotent class of $\mathfrak{M}$ is $<p$. 

The operator $B$ induces the operator on $\mathfrak{M}$ which we denote also by $B$. 
Clearly, $B=\wt{\exp}\,\c B$ where $\c B$ is a differentiation on $\mathfrak{M}$ such that 
for all $i\geqslant 1$, $\c B(\mathfrak{M}(i))\subset \mathfrak{M}(i+1)$. 

Let $\wt{\mathfrak M}$ be a semi-direct product of $\mathfrak{M}$ and the trivial 
Lie algebra $(\Z /p^M)w$ via $\c B$. This means 
that $\wt{\fr{M}}=\fr{M}\oplus (\Z /p^M)w$ as $\Z /p^M$-module, 
$\mathfrak{M}$ and $(\Z /p^M)w$ are Lie subalgebras of 
$\wt{\mathfrak{M}}$ and for any $m\in\mathfrak{M}$, $[m,w]=\c B(m)$. 
Clearly, 
$C_2(\wt{\mathfrak{M}})=[\wt{\mathfrak{M}},\wt{\mathfrak{M}}]\subset \mathfrak{M}(2)$.  
This implies that   $\wt{\mathfrak{M}}$ 
has nilpotent class $<p$ and we can consider the 
$p$-group $G(\wt{\mathfrak{M}})$. This group has nilpotent class $<p$ and period $p^M$ 
(because for any $\bar m\in\wt{\mathfrak{M}}$, its $p^M$-th power in $G(\wt{\mathfrak{M}})$ 
equals $p^M\bar m=0$). 

Note that the conjugation by $w$ in $G(\wt{\mathfrak{M}})$ is given by the automorphism 
$\wt{\exp}\,\c B=B$. Indeed, 
if $m\in\mathfrak{M}$ then 
$$B(m)=(\wt{\exp}\c B)m=\sum _{0\leqslant n<p}\c B^n(m)/n!=(-w)\circ m\circ w $$
(use very well-known formula in a free associative algebra $\Q [[X,Y]]$, 
$$\exp (-Y)\exp (X)\exp (Y)=\exp (X+\ldots +(\ad\,^nY)X/n!+\ldots )\, ,$$
where $\ad\,Y:X\mapsto [X,Y]$). 

In particular, for any element   
$\bar m=m\,\op{mod}\,\c N(p)\in \mathfrak{M}$, 
we have $w_1\circ \bar m=B(\bar m)\circ w_1$, where $w_1=-w$. Therefore, 
$0=(\bar m\circ w_1)^{p^M}=\bar m\circ B(\bar m)\circ\dots \circ B^{p^M-1}(\bar m)\circ w_1^{p^M}$, 
and it remains to note that $w_1^{p^M}=0$. 
\end{proof} 

Applying the above Lemma we obtain that 
$$c^0\circ (A^0\otimes h^{-1})c^0\circ 
 \dots \circ (A^0\otimes h^{-1})^{p^M-1}c^0\in\c N(p)\subset S^{p-1}\c M$$ 
and, therefore, 
$h^{0p^M}_{<p}(\bar f)=0$.  

Thus, we proved that $\wt{\c G}^{p^M}_{h}C_p(\wt{\c G}_{h})\subset\c H$. 
\medskip

Suppose $g=h_{<p}^m\,l\in\c H$ with some $l\in G(\c L)$. Then 
$g(f)=b\circ f$ where $b\in S^{p-1}\c M_{<p}$. Note that  
$\sigma (b)\in  S^{p-1}\c M_{<p}$. Then  
$$g(e)\circ b\circ f=g(e)\circ g(f)=g(\sigma f)=
\sigma b\circ \sigma f=\sigma b\circ e\circ f$$ 
implies that $g(e)\equiv e\,\op{mod}\, S^{p-1}\c M$. 
Thus $(\id\otimes h)^m(e)\equiv e\,\op{mod}\,S^{p-1}\c M$. 

Now use that $e\equiv e^{\dag }\,\op{mod}\,
\c L_{\m(\c K)}+C_2(\c L)_{\c K}$, cf. the beginning 
of the proof of Proposition \ref{P2.6}. 

Clearly, $\c L_{\m (\c K)}+\c L(p)_{\c K}\supset S^{p-1}\c M$ 
and, therefore, for the element 
$$e^{\dag }_{<p}:=\sum _{i,b}\sum _{1\leqslant m<p}
\frac{t^b}{S^m}\beta _iV_{(b,m),i}$$ 
we obtain 
 $(\id _{\c L}\otimes h)^m(e^{\dag }_{<p})
\equiv e^{\dag }_{<p}\,\op{mod}\,\c L_{\m (\c K)}+C_2(\c L_{\c K})$. 
But 
$$h^m(e^{\dag }_{<p})\equiv \sum _{i,b}\sum _{1\leqslant m<p}
\frac{t^bE(bm, S)}{S^m}\beta _iV_{(b,m),i}\,\op{mod}\,\c L_{\m (\c K)}
+\c L(p)_{\c K}$$
Now following the coefficients for $V_{(b,p-2),i}$ we obtain  
$m\equiv 0\,\op{mod}\,p^M$. Therefore, $l\in\c H\cap G(\c L)=C_p(\wt{\c G}_h)$ 
and $\c H\subset \wt{\c G}^{p^M}_{h}C_p(\wt{\c G}_{h})$. 

Finally, we have $\wt{\c G}_h/\c H=\c G_h$,  
$\c H\,\op{mod}\,C_p(\wt{\c G}_h)=\langle h_{<p}^{p^M}\rangle $ and, therefore, 
$\op{Coker}\,\kappa =\langle h\rangle \,\op{mod}\,\langle h^{p^M}\rangle $. 
\end{proof} 

\begin{Cor} \label{C3.7} 
If $L_h$ is a Lie $\Z /p^M$ algebra such that $\c G_{h}=G(L_h)$ then 
\eqref{E3.5} induces 
the following short exact sequence of Lie $\Z /p^M$-algebras 
$$0\To \c L/\c L(p)\To L_h\To (\Z /p^M)h\To 0$$
\end{Cor} 

\begin{remark}
 In \cite{Ab11} we studied the structure of 
the above Lie algebra $L_h$ in the case $M=1$. 
The case of arbitrary $M$ will be considered in a forthcoming paper. 
\end{remark}

\subsection{Ramification estimates} \label{S3.5} 
Use the identification from Subsection \ref{S1.3},  
$\eta _M:\Gal (\c K_{<p}/\c K)=\c G_{<p}\simeq G(\c L)$ and set for all 
for $s\in\N $, $\c K[s,M]:=\c K_{<p}^{G(\c L(s+1))}$. 
Denote by 
$v(s,M)$ the maximal upper ramification number of the extension 
$\c K[s,M]/\c K$. In other words, 
$$v(s,M)=\max\{v\ |\ \c G_{<p}^{(v)} \text{ acts non-trivially 
on }\c K[s,M]\}\,.$$

\begin{Prop} \label{P3.8}
For all $s\in\N $, $v(s,M)=p^{M-1}(e^*s-1)$ 
(for the definition of $e^*$ cf, Subsection \ref{S2.1}).  
\end{Prop}

\begin{proof}
 Recall, cf. Subsection \ref{S1.7}, that for any $v\geqslant 0$, 
the ramification subgroups $\c G_{<p}^{(v)}$ are identified 
with the ideals $\c L^{(v)}$ of $\c L$, and for sufficiently large 
$N=N(v)$, the ideal 
 $\c L^{(v)}_k$ is generated by all $\sigma ^n\c F^0_{\gamma ,-N}$, 
 where $\gamma\geqslant v$, $n\in\Z /N_0$ and the elements 
 $\c F^0_{\gamma ,-N}$ are given by \eqref{E1.9}. 

Let $e^0=e^*(1-1/p)$. 

\begin{Lem} \label{L3.9} 
 If $a\in\Z ^+(p)$, $u\in\N $ and $0\leqslant c<M$ then the following 
two conditions are equivalent: 
\medskip 

{\rm a)} $t^aS^{-u}\in\m (\c K)\,\op{mod}\,p^{c}\,O(\c K)$;
\medskip 

{\rm b)}  $a>e^*u+e^0(c-1)$. 
\end{Lem}

\begin{proof}[Proof of lemma]  
Proposition \ref{P2.1}c) implies that 
$$t^aS^{-u}=t^{a-ue^*}\eta _0\left (1+\sum _{i\geqslant 1} 
t^{-ie^0}\eta _i(u)p^i\right )$$
where $\eta _0$ and all $\eta _i(u)$ are invertible elements of 
$W_M(k)[[t]]\subset O(\c K)$. Therefore, 
$t^aS^{-u}\in\m (\c K)\,\op{mod}\,p^cO(\c K)$ if and only if 
for all $1\leqslant i<c$, $t^{a-ue^*-ie^0}\in\m (\c K)$, i.e. 
$a-ue^*-(c-1)e^0>0$. The lemma is proved. 
\end{proof} 

\begin{Cor} \label{C3.10} $D_{an}\in\c L(u)_k\,\op{mod}\,p^cO(\c K)$ 
if and only if we have that  
$a\geqslant e^*(u-1)+(c-1)e^0+1$. 
 
\end{Cor}

\begin{Lem}\label{L3.11} 
Suppose $N\geqslant 0$. 
\medskip 

 {\rm a)} If $\gamma >p^{M-1}(e^*s-1)$ then $\c F^0_{\gamma ,-N}\in\c L(s+1)_k$;
\medskip 

{\rm b)} if $\gamma =p^{M-1}(e^*s-1)$ then 
$$\c F^0_{\gamma ,-N}\equiv p^{M-1}D_{e^*s-1,M-1}
\,\op{mod}\,\c L(s+1)_k\,.$$
\end{Lem}

\begin{proof}[Proof of lemma] 
For any $\gamma >0$, $\c F^0_{\gamma , -N}$ is a $\Z /p^M$-linear combination 
of the monomials of the form 
$$X(b\,; a_1,\dots ,a_r;m_2,\dots ,m_r)=p^ba_1[\dots [D_{a_1,b-m_1},D_{a_2,b-m_2}],\dots ,
D_{a_r,b-m_r}]\,,$$
where $0\leqslant b<M$, $1\leqslant r<p$, all $a_i\in\Z ^0(p)$, 
$0=m_1\leqslant m_2\leqslant\dots \leqslant m_r$,  and 
$$p^b\left (a_1+\frac{a_2}{p^{m_2}}+\dots +\frac{a_r}{p^{m_r}}\right )=\gamma \,.$$ 
For $1\leqslant i\leqslant r$, let $u_i\in\Z $ be such that (note that $p^M|e^*$, $p^{M-1}|e^0$ and 
if $M=1$ then $M-b-1=0$) 
$$1+e^*(u_i-1)+e^0(M-b-1)\leqslant a_i<e^*u_i+e^0(M-b-1)\,.$$
This means that all $D_{a_i,b-m_i}\in\c L(u_i)_k\,\op{mod}\,p^{M-b}\c L_{k}$\,. 

Suppose $X(b\,; a_1,\dots ,a_r;m_2,\dots ,m_r)\notin \c L(s+1)_k$. This implies that  
$u_1+\dots +u_r\leqslant s$ and, therefore, $a_1+\dots +a_r\leqslant 
e^*s+re^0(M-b-1)-r$. 

If $\gamma >p^{M-1}(e^*s-1)$ then $a_1+\dots +a_r>p^{M-b-1}(e^*s-1)$ and 
$$e^*s+re^0(M-b-1)-r>p^{M-b-1}(e^*s-1).$$

Set $c=M-b-1$, then $0\leqslant c<M$ and 
$$(p^c-1)(e^*s-1)\leqslant r(e^0c-1)\,.$$

If $c=0$ then $r\leqslant 0$, contradiction.

If $c\geqslant 1$ then (use that $r\leqslant p-1$ and $s\geqslant 1$) 
$$(1+p+\dots +p^{c-1})(e^*-1)\leqslant e^0c-1\,.$$
But then $e^*=e^0(1+1/(p-1))\geqslant e^0+1$ implies that $1+p+\dots +p^{c-1}<c$. This 
contradiction proves a).

Suppose $\gamma =p^{M-1}(e^*s-1)$. Then the expression for $\c F^0_{\gamma ,-N}$ contains the term 
$p^{M-1}D_{e^*s-1,M-1}$. Take (with above notation) any another monomial 
$X(b; a_1,\dots ,a_r;m_2,\dots ,m_r)$ from the expression of 
$\c F^0_{\gamma ,-N}$. Clearly,  $r\geqslant 2$. As earlier, 
the assumption that this monomial does not belong to 
$\c L(s+1)_k$ implies that 
$$(p^c-1)(e^*s-1)\leqslant r(e^0c-1)+1\,.$$

If $c=0$ then $r\leqslant 1$, contradiction. 

If $c\geqslant 1$ then again use that $r\leqslant p-1$ to obtain 
$$(1+p+\dots +p^{c-1})(e^*s-1)\leqslant e^0c-1+1/(p-1)<e^0c$$
and note that the left-hand side of this inequality $>ce^0$ (use that 
$e^*s-1\geqslant e^*-1\geqslant e^0$). The contradiction. 
The lemma is completely proved.  
\end{proof} 

It remains to note that Lemma \ref{L3.11} implies that 
$$\op{max}\{v\ |\ \c L^{(v)}\not\subset\c L(s+1)\}=p^{M-1}(e^*s-1)\,.$$
Proposition \ref{P3.8} is completely proved. 
\end{proof} 

\section{Applications to the mixed characteristic case} \label{S4} 

Let $K$ be a finite field extension of $\Q _p$ with the residue field 
$k\simeq \F _{p^{N_0}}$ and the ramification index $e_K$. Let $\pi _0$ be 
a uniformising element in $K$. Denote by $\bar K$ an algebraic closure of $K$ and set 
$\Gamma =\Gal (\bar K/K)$. Assume that $K$ contains a primitive 
$p^M$-th root of unity $\zeta _M$.

\subsection{} \label{S4.1} 
For $n\in\N $, choose $\pi _n\in\bar K$ such that $\pi _n^p=\pi _{n-1}$. 
Let $\wt{K}=\bigcup _{n\in\N }K(\pi _n)$,  
$\Gamma _{<p}:=\Gamma /\Gamma ^{p^M}C_p(\Gamma )$ and $\wt{\Gamma }=
\Gal (\bar K/\wt{K})$. Then $\wt{\Gamma }\subset\Gamma $ 
induces a continuous group homomorphism 
$i :\wt{\Gamma }\To \Gamma _{<p}$.  

We have $\Gal (K(\pi _M)/K)=\langle \tau _0\rangle ^{\Z /p^M}$, where $\tau _0(\pi _M)
=\pi _M\zeta _M$. Let $j:\Gamma _{<p}\To \Gal (K(\pi _M)/K)$ be a natural epimorphism. 

\begin{Prop} \label{P4.1}
 The following sequence 
$$\wt{\Gamma }\overset{i}\To \Gamma _{<p}\overset{j}
\To \langle \tau _0\rangle ^{\Z /p^M}\To 1$$ 
is exact. 
\end{Prop}

\begin{proof}
For $n>M$, let $\zeta _n\in\bar K$ be such that $\zeta _n^{p}=\zeta _{n-1}$.  

Consider $\wt{K}'=\bigcup _{n\geqslant M}K(\pi _n,\zeta _n)$. Then 
$\wt {K}'/K$ is Galois with the Galois group 
$\Gamma _{\wt{K}'/K}=\langle \sigma, \tau \rangle $. Here for any $n\geqslant M$ and 
some $s_0\in \Z$, 
$\sigma \zeta _n=\zeta _n^{1+p^Ms_0}$, 
$\sigma\pi _n=\pi _n$, $\tau (\zeta _n)=\zeta _n$, $\tau\pi _n=\pi _n\zeta _n$ and 
$\sigma ^{-1}\tau\sigma =\tau ^{(1+p^Ms_0)^{-1}}$. 

Therefore, $\Gamma _{\wt{K}'/K}^{p^M}=\langle \sigma ^{p^M},\tau ^{p^M}\rangle $ and 
for the subgroup of second commutators we have 
$C_2(\Gamma _{\wt{K}'/K})\subset 
\langle\tau ^{p^M}\rangle\subset\Gamma _{\wt{K}'/K}^{p^M}$. This implies that  
$$\Gamma ^{p^M}_{\wt{K}'/K}C_p(\Gamma _{\wt{K}'/K})=
\langle \sigma ^{p^M}, \tau ^{p^M}\rangle $$ 
 and 
for $\Gamma _{\wt{K}'/K}(M):=\Gamma _{\wt{K}'/K}/
\Gamma ^{p^M}_{\wt{K}'/K}C_p(\Gamma _{\wt{K}'/K})$,  
we obtain a natural exact sequence 
$$\langle\sigma\rangle \To \Gamma _{\wt{K}'/K}(M)\To 
\langle\tau\rangle 
\,\op{mod}\,\langle\tau ^{p^M}\rangle =\langle\tau _0\rangle ^{\Z /p^M}\To 1\,.$$
Note that  $\Gamma _{\wt{K}'}$ together with a lift $\hat\sigma \in\wt{\Gamma}$ of $\sigma $  
generate $\wt{\Gamma }$. 
The above short exact sequence implies that 
$\op{Ker}
\left (\Gamma _{<p}\To \langle\tau _0\rangle ^{\Z /p^M} \right )$
is generated by $\hat\sigma $ and the image of $\Gamma _{\wt{K}'}$. 
So, this kernel coincides with the image of $\wt{\Gamma }$ in $\Gamma _{<p}$.  
\end{proof}

\subsection{} \label{S4.2} 
Let $R$ be Fontaine's ring. We have a natural embedding 
$k\subset R$ and an element $t_0=(\pi _n\,\op{mod}\,p)_{n\geqslant 0}\in R$. 
Then we can identify the field  $k((t_0))$ with the field $\c K$ from 
Sections \ref{S1}-\ref{S3}. If $R_0=\op{Frac}\,R$ 
then $\c K$ is a closed subfield of $R_0$ and 
the theory of the field-of-norms functor identifies $R_0$ with 
the completion of the separable closure $\c K_{sep}$ of $\c K$ in $R_0$. 
Note that $R$ is the valuation ring of $R_0$ and denote by $\m _R$ the maximal ideal of $R$.

This allows us to identify $\c G=\op{Gal} (\c K_{sep}/\c K)$ 
with $\wt{\Gamma }\subset \Gamma \subset \Aut\, R_0$. This identification 
is compatible with the appropriate ramification filtrations. Namely, 
if $\varphi _{\wt{K}/K}$ is the Herbrand function of the (arithmetically profinite) 
field extension $\wt{K}/K$ then for any $v\geqslant 0$, 
$\c G^{(v)}=\Gamma ^{(v_1)}\cap\wt{\Gamma}$, where $v_1=
\varphi _{\wt{K}/K}(v)$. 

Let as earlier,  $\c G_{<p}=\c G/\c G^{p^M}C_p(\c G)$. 
Then the embedding $\c G=\wt{\Gamma }\subset\Gamma $ induces a natural continuous 
morphism $\iota $ of the infinite group $\c G _{<p}$ 
to the finite group $\Gamma _{<p}$. Therefore, by 
Proposition \ref{P4.1} we obtain the following 
exact sequence 

\begin{equation} \label{E4.1} 
\c G_{<p}\overset{\iota}\To \Gamma _{<p}
\overset{j}\To \langle\tau _0\rangle ^{\Z /p^M}\To 1\,.
\end{equation}  

Let $\zeta _M=1+\sum _{i\geqslant 1}[\beta _i]\pi _0^i$ with all $\beta _i\in k$. 
Consider the identification of rings $R/t_0^{e_K}\simeq O_{\bar K}/p$ given by 
$(r_0,\dots ,r_n,\dots )\mapsto r_0$. If $\varepsilon =(\zeta _n)_{n\geqslant 0}$ 
is Fontaine's element such that $\zeta _M$ is our fixed $p^M$-th root of unity then 
we have in $W_M(R)$ the following congruence (as earlier, $t=(t_0,\dots, 0)\in W_M(R)$) 
\begin{equation}\label{E4.2} 
 \sigma ^{-M}\varepsilon \equiv 1+\sum _{i\geqslant 1}\beta _it^i\,\op{mod}\,(t^{e_K},p)\,.
\end{equation}
Now we can specify the choice of the elements 
$S_0, S\in\m (\c K)$, cf. Subsection \ref{S2.1}, 
by setting $E(1,S_0)=1+\sum _{i}\beta _it^i$ and $S=[p]^M(S_0)$. Note that 
$S\,\op{mod}\,p$ generates the ideal $(t_0^{e^*})$ in $O_{\c K}=k[[t_0]]$, 
where $e^*=pe_K/(p-1)$. Now congruence \eqref{E4.2} 
can be rewritten in the following form 
$$\sigma ^{-M}\varepsilon \equiv E(1,S_0)\,\op{mod}\,(\sigma ^{-1}S^{p-1}, p)\,.$$ 
Applying $\sigma $ we obtain 
$$\sigma ^{-M+1}\varepsilon \equiv E(1,[p]S_0)\,\op{mod}\,(S^{p-1},p)\,,$$ 
and then taking $p^{M-1}$-th power 
$$\varepsilon \equiv E(1,S)\,\op{mod}\,S^{p-1}W_M(R)\,.$$

\subsection{} \label{S4.3} 
Let $v_{\c K}$ be the extension of the normalized valuation 
on $\c K$ to $R_0$.  
Consider a continuous field embedding $\eta _0:\c K\To R_0$ compatible with 
$v_{\c K}$. Denote by $\op{Iso}(\eta _0,\c K_{<p},R_0)$ the set of all 
extensions $\eta _{<p,0}$ of $\eta _0$ to $\c K_{<p}$. 
This set is a principal homogeneous space over $\c G_{<p}=G(\c L)$.

Choose a lift $\eta :O(\c K)\To W_M(R_0)$ such that 
$\eta\,\op{mod}\,p=\eta _0$ and $\eta\sigma =\sigma\eta $. Proceeding similarly 
to Subsection \ref{S1.1} we can identify the set of all lifts 
$\eta _{0,<p}$ of $\eta _0$ from $\op{Iso}(\eta _0,\c K_{<p}, R_0)$ with 
the set of all (commuting with $\sigma $) lifts 
$\eta _{<p}$ of $\eta $ from  
$\op{Iso}(\eta ,O(\c K_{<p}), W_M(R_0))$.

Specify uniquely each lift $\eta _{<p}$ by the knowledge of  
$\eta _{<p}(f)\in \c L_{R_0}$ in the set of all solutions $f'\in\c L_{R_0}$ 
of the equation  $\sigma f'=\eta (e)\circ f'$. 
(The elements $e\in\c L_{\c K}$ and $f\in\c L_{\c K_{<p}}$ 
were chosen in Subsection \ref{S1.4}.) 

Consider the appropriate submodules $\c M\subset\c L_{\c K}$, 
$\c M_{<p}\subset\c L_{\c K_{<p}}$ from Subsection \ref{S3.4} and 
define similarly 
$$\c M_{R_0}=
\sum _{1\leqslant s<p}S^{-s}\c L(s)_{\m (R)}+\c L(p)_{R_0}\subset\c L_{R_0}\,,$$ 
where $\m (R)=W_M(\m _R)$. 
We know that $e\in \c M$, $f\in \c M_{<p}$ and for 
similar reasons, all $\eta _{<p}(f)\in\c M_{R_0}$. 

\begin{Lem} \label{L4.3} With above notation suppose that 
$\eta (e)\equiv e\,\op{mod}\,S^{p-1}\c M_{R_0}$. 
Then there is $c\in S^{p-1}\c M_{R_0}$ such that 
$\eta (e)=\sigma c\circ e\circ (-c)$. 
\end{Lem}

\begin{proof} 
 Note that $S^{p-1}\c M_{R_0}$ is an ideal in $\c M_{R_0}$ 
and for any $i\in\N$ and $m\in S^{p-1}C_i(\c M_{R_0})$, 
there is $c\in S^{p-1}C_i(\c M_{R_0})$ 
such that $\sigma c-c=m$. 
(Use that $\sigma $ is topologically nilpotent on $S^{p-1}C_i(\c M_{R_0})$.)

Therefore,  there is $c_1\in S^{p-1}\c M_{R_0}$ such that 
$\eta (e)=e+\sigma c_1-c_1$. This implies that 
$\eta (e)\circ c_1\equiv \sigma c_1\circ e\,
\op{mod}\,S^{p-1}C_2(\c M_{R_0})$. Similarly, there is $c_2\in 
S^{p-1}C_2(\c M_{R_0})$ such that 
$\eta (e)\circ c_1+c_2=\sigma c_2+\sigma c_1\circ e_0$ and 
$\eta (e_0)\circ c_1\circ c_2\equiv \sigma c_2\circ \sigma c_1\circ e_0 
\op{mod}\,S^{p-1}C_3(\c M_{R_0})$, and so on. 

After $p-1$ iterations we obtain for $1\leqslant i<p$ the elements 
$c_i\in S^{p-1}C_i(\c M_{R_0})$ such that 
$$\eta (e)\circ (c_1\circ \dots \circ c_{p-1})=
\sigma (c_{p-1}\circ \dots \circ c_1)\circ e.$$
The lemma is proved.
\end{proof}

The above lemma implies the following properties:
\medskip 

\begin{Prop} \label{P4.3} 
 {\rm a)}\ If $\eta (e)\equiv e\,\op{mod}\,S^{p-1}\c M_{R_0}$ then 
for any $\eta _{<p}\in\op{Iso} (\eta ,\c K_{<p}, R_0)$, there is 
a unique $l\in G(\c L)\,\op{mod}\,G(\c L(p))$ such that 
$$\eta _{<p}(f)\equiv f\circ l\,\op{mod}\,S^{p-1}
\c M_{R_0}\,.$$

{\rm b)} Suppose 
$\eta ',\eta '':O(\c K)\To W_M(R_0)$ are such that 
$$\eta '(t)\equiv\eta ''(t)\,\op{mod}\,
S^{p-1}W_M(\m _{R})\,.$$ 
If 
$\eta '_{<p}\in\op{Iso}(\eta ',O(\c K_{<p}),W_M(R_0))$ and 
$\eta ''_{<p}\in\op{Iso}(\eta '',O(\c K_{<p}),W_M(R_0))$ then there is a unique 
$l\in G(\c L)$ such that  
$$\eta '_{<p}(f)\equiv \eta ''_{<p}(f)\circ l\,\op{mod}\,S^{p-1}\c M_{R_0}.$$ 
\end{Prop}

\subsection{} \label{S4.4} 

The action of $\Gamma =\op{Gal}(\bar K/K)$ on $R_0$ 
is strict and, therefore, the elements $g\in\Gamma $ 
can be identified with all continuous field embeddings 
$g:\c K_{sep}\to R_0$ such that $g|_{\c K}$ belongs to 
the set $\langle \tau _0\rangle =\{\tau _0^{a}\ |\ a\in\Z _p\}$.

Extend 
$\tau _0$ now to a continuous embedding 
$\tau :O(\c K)\To W_M(R_0)$ uniquely determined by the condition 
$\tau (t)=t\varepsilon $. Clearly, $\tau $ 
commutes with $\sigma $. Then the results of Subsection \ref{S1.1} 
imply that the elements of $\Gamma $ are identified with 
the continuous embeddings $g: O(\c K_{sep})\to W_M(R_0)$ such that 
$g|_{O(\c K)}$ belongs to the set $\langle \tau \rangle $. 

Consider $h_0\in\op{Aut}(\c K)$ such that 
$h_0(t_0)=t_0E(1,S\,\op{mod}\,p)$  and $h_0|_k=\id $. Then 
its lift $h\in \op{Aut} O(\c K)$ such that $h(t)=tE(1,S)$ commutes with 
$\sigma $ and there are the appropriate 
groups $\wt{\c G}_h$ and $\c G_h$ from Section \ref{S3}.

 Clearly,  
$h(t)\equiv \tau (t)\,\op{mod}\,S^{p-1}\m _{R}$ and we can apply 
Proposition \ref{P4.3}b). This implies that 
the $\Gamma $-orbit of $f\,\op{mod}\,S^{p-1}\c M_{R_0}$ 
is contained in the  
$\wt{\c G}_{h}$-orbit of $f\,\op{mod}\,S^{p-1}\c M_{R_0}$.   
Therefore, there is a map of sets $\kappa :\Gamma \To \c G_h$ 
uniquely determined by the requirement that  for any $g\in\Gamma $, 
$$(\id _{\c L}\otimes g)f\equiv (\id _{\c L}\otimes \kappa (g))f\,
\op{mod}\,S^{p-1}\c M_{R_0}\, .$$
(Use that $\c G_h$ strictly acts on the $\wt{\c G}_h$-orbit of 
$f\,\op{mod}\,S^{p-1}\c M_{R_0}$.)

\begin{Prop} \label{P6.4} 
$\kappa $ induces a group isomorphism $\kappa _{<p}:\Gamma _{<p}\To\c G_h$. 
 \end{Prop}

\begin{proof} Suppose $g_1,g\in\Gamma $. 
Let $c\in\c L_{\c K}$ and $A\in\Aut\,{\c L}$ be such that 
 $(\id _{\c L}\otimes \kappa (g))f=c\circ (A\otimes \id _{\c K_{<p}})f$. Then 
we have the following congruences modulo $S^{p-1}\c M_{R_0}$
$$(\id _{\c L}\otimes\kappa (g_1g))f\equiv (\id _{\c L}\otimes g_1g)f
\equiv (\id _{\c L}\otimes g_1)(\id _{\c L}\otimes g)f\equiv $$
$$(\id _{\c L}\otimes g_1)(\id _{\c L}\otimes\kappa (g))f\equiv 
(\id _{\c L}\otimes g_1)(c\circ (A\otimes \id _{\c K_{<p}})f)\equiv $$
$$(\id _{\c L}\otimes g_1)c\circ (A\otimes g_1)f
\equiv (\id _{\c L}\otimes \kappa (g_1))c\circ (A\otimes\kappa (g_1))f\equiv $$
$$(\id _{\c L}\otimes \kappa (g_1))(c\circ (A\otimes  \id _{\c K_{<p}})f)\equiv 
(\id _{\c L}\otimes \kappa (g_1))(\id _{\c L}\otimes\kappa (g))f$$
$$\equiv (\id _{\c L}\otimes \kappa (g_1)\kappa (g))f$$
and, therefore, $\kappa (g_1g)=\kappa (g_1)\kappa (g)$ 
(use that $\c G_h$ acts strictly on the orbit of $f$).  

Therefore, 
$\kappa $ factors through the natural projection $\Gamma \to\Gamma _{<p}$  
and defines the group homomorphism $\kappa _{<p}:\Gamma _{<p}\to\c G_h$. 

Recall that we have the field-of-norms identification 
$\wt{\Gamma }=\c G$ and, therefore, $\kappa _{<p}$ identifies   
the groups $\kappa (\wt{\Gamma })$ and $G(\c L/\c L(p))\subset\c G_h$. 
Besides, $\kappa $ induces a group isomorphism of 
$\langle \tau _0\rangle ^{\Z /p^M}$ and $\langle h_0\rangle ^{\Z /p^M}$. 
Now Proposition \ref{P4.1} 
implies that $\kappa _{<p}$ is isomorphism. 
\end{proof}

Under the isomorphism $\kappa _{<p}$, 
the subfields $\c K[s,M]\subset \c K_{<p}$, where $1\leqslant s<p$ (cf. Subsection \ref{S3.5}),  
give rise to the subfields $K[s,M]\subset K_{<p}$ such that 
$\op{Gal }(K[s,M]/K)=\Gamma /\Gamma ^{p^M}C_{s+1}(\Gamma )$. In other words, 
the extensions $K[s,M]$ appear as the maximal $p$-extensions of 
$K$ with the Galois group of period $p^M$ and nilpotent class $s$. 

Using that the identification $\c G=\wt{\Gamma }$ is compatible 
with ramification filtrations, cf. Subsection \ref{S4.2}, we obtain the following 
result about the maximal upper ramification numbers of the field extensions $K[s,M]/K$, 
where $M\in\N $ and $1\leqslant s<p$. 

\begin{Thm} \label{Th4.4} If $[K:\Q _p]<\infty $, $e_K$ is the ramification index of 
$K$ and $\zeta _M\in K$ then for $1\leqslant s<p$, 
$$v(K[s,M]/K)=e_K\left (M+\frac{s}{p-1}\right )-\frac{1-\delta _{1s}}{p}\,.$$
\end{Thm}

\begin{proof} Note first, that the Herbrand function 
$\varphi _{\wt{K}/K}(x)$ is continuous for all $x\geqslant 0$, $\varphi _{\wt{K}/K}(0)=0$ 
and its derivative $\varphi '_{\wt{K}/K}$ equals 1 if $x\in (0,e^*)$ 
and equals $p^{-m}$, if $m\in\N $ and $x\in (e^*p^{m-1}, e^*p^m)$. 

From Proposition \ref{P3.8} we obtain that 
$$v(K[s,M]/K)=\max\left\{ v(K(\pi _M)/K), 
\varphi _{\wt{K}/K}(p^{M-1}(se^*-1))\right\}\,.$$ 

Note that $v(K(\pi _M)/K)=\varphi _{\wt{K}/K}(p^{M-1}e^*)=
e^*+e_K(M-1)$ and, therefore,   
$$v(K[1,M]/K)=v(K(\pi _M)/K)=
e_K\left (M+\frac{1}{p-1}\right )\,.$$

If $2\leqslant s<p$ then $v(K[s,M/K)$ equals 
$\varphi _{\wt{K}/K}(p^{M-1}(se^*-1))=$
$$=\varphi _{\wt{K}/K}(p^{M-1}e^*)+
\frac{p^{M-1}(se^*-1)-p^{M-1}e^*}{p^M}=e_K\left (M+\frac{s}{p-1}\right )-\frac{1}{p}\,.$$

\end{proof}

\end{document}